\documentclass[11pt,oneside]{amsart}

\usepackage{amsmath,ifthen, amsfonts, amssymb,
srcltx, amsopn, color, enumerate, hyperref}
\usepackage[cmtip,arrow]{xy}
\usepackage{pb-diagram, pb-xy}
\usepackage{overpic}
\usepackage{hyperref}
\hypersetup{colorlinks}
\usepackage{tikz} 
\usepackage[T1]{fontenc}

\newcommand{\showcomments}{yes}

\newsavebox{\commentbox}
{\ifthenelse{\equal{\showcomments}{yes}}%
{\footnotemark
        \begin{lrbox}{\commentbox}
        \begin{minipage}[t]{1.25in}\raggedright\sffamily\tiny
        \footnotemark[\arabic{footnote}]}
{\begin{lrbox}{\commentbox}}}%
{\ifthenelse{\equal{\showcomments}{yes}}%
{\end{minipage}\end{lrbox}\marginpar{\usebox{\commentbox}}}
{\end{lrbox}}}

\newcounter{ax}
\newtheorem{thm}{Theorem}[section]

\newtheorem{lem}[thm]{Lemma}

\newtheorem{cor}[thm]{Corollary}

\newtheorem{prop}[thm]{Proposition}

\newtheorem{thmi}{Theorem}
\newtheorem{cori}[thmi]{Corollary}

\theoremstyle{definition}
\newtheorem{defn}[thm]{Definition}
\newtheorem{rem}[thm]{Remark}

\newtheorem{question}[thm]{Question}
\newtheorem{claim}{Claim}

\newtheorem{claim*}{Claim}

\DeclareMathOperator{\dimension}{dim}

\DeclareMathOperator{\Aut}{Aut}
\DeclareMathOperator{\Out}{Out}

\DeclareMathOperator{\stabilizer}{Stab}

\newcommand{\neb}{\mathcal N}

\newcommand{\field}[1]{\mathbb{#1}}
\newcommand{\integers}{\ensuremath{\field{Z}}}

\newcommand{\naturals}{\ensuremath{\field{N}}}

\newcommand{\interior} [1] {{\ensuremath \text{\rm Int}(#1) }}

\makeatletter

\newcommand{\Rmnum}[1]{\mathbf{{\expandafter\@slowromancap\romannumeral #1@}}}

\makeatother

\let\oldmarginpar\marginpar
\renewcommand\marginpar[1]{\-\oldmarginpar[\raggedleft\footnotesize #1]%
{\raggedright\footnotesize #1}}

\newcounter{enumitemp}

\newcommand{\dist}{\textup{\textsf{d}}}

\newcommand{\OR}{\overrightarrow}

\newcommand{\half}{{\frac 1 2}}
\newcommand{\interval}[3]{\left#1 #2 \right#3}

\newcommand{\onto}{\twoheadrightarrow}
\newcommand{\defRet}{\stackrel{\supset}{\onto}}

\newcommand{\bur}[1]{\mathcal{B}\left(#1\right)}

\newcommand{\faces}{\mathrm{Faces}}
\begin{document}
\title[Panel collapse]{Panel collapse and its applications}
\date{\today}

\author[M.F.~Hagen]{Mark F. Hagen}
\address{School of Mathematics, University of Bristol, Bristol, UK}
\email{markfhagen@gmail.com}

\author[N.W.M~Touikan]{Nicholas W.M. Touikan}
\address{Department of Mathematics and Statistics, University of New Brunswick, Fredericton, New Brunswick, Canada}
\email{ntouikan@unb.ca}

\maketitle

\begin{abstract}
We describe a procedure called \emph{panel collapse} for replacing a CAT(0) cube 
complex $\Psi$ by a ``lower complexity'' CAT(0) cube complex $\Psi_\bullet$ 
whenever $\Psi$ contains a codimension--$2$ hyperplane that is \emph{extremal} 
in one of the codimension--$1$ hyperplanes containing it.  Although 
$\Psi_\bullet$ is not in general a subcomplex of $\Psi$, it is a subspace 
consisting of a subcomplex together with some cubes that sit inside $\Psi$ 
``diagonally''.  The hyperplanes of $\Psi_\bullet$ extend to hyperplanes of 
$\Psi$.  Applying this procedure, we prove: if a group $G$ acts cocompactly on 
a CAT(0) cube complex $\Psi$, then there is a CAT(0) cube complex $\Omega$ so 
that $G$ acts cocompactly on $\Omega$ and for each hyperplane 
$H$ of $\Omega$, the stabiliser in $G$ of $H$ acts on $H$ essentially.  

Using panel collapse, we obtain a new proof of Stallings's theorem on
groups with more than one end.  As another illustrative example, we
show that panel collapse applies to the exotic cubulations of free
groups constructed by Wise in~\cite{Wise:recube}.  Next, we show that
the CAT(0) cube complexes constructed by Cashen-Macura
in~\cite{cashen-macura} can be collapsed to trees while preserving all
of the necessary group actions.  (It also illustrates that
our result applies to actions of some non-discrete groups.)  We also discuss possible applications to quasi-isometric rigidity for certain classes of graphs of free groups with cyclic edge groups. Panel collapse is also used in forthcoming work of the
first-named author and Wilton to study fixed-point sets of finite
subgroups of $\Out(F_n)$ on the free splitting complex.  Finally, we apply panel collapse to a conjecture of Kropholler, obtaining a short proof under a natural extra hypothesis.
\end{abstract}

\setcounter{tocdepth}{1}
\tableofcontents

\section*{Introduction}\label{sec:intro}
CAT(0) cube complexes, which generalise simplicial trees in several ways, have 
wide utility in geometric group theory; making a group act by isometries on a 
CAT(0) cube complex can reveal considerable information about the 
structure of the group.  The nature of this information depends on where the 
action lies along a ``niceness spectrum'', with (merely) \emph{fixed-point 
freely} at one end, and \emph{properly and cocompactly} at the other.

In this paper, we focus on cocompact (but not necessarily proper) actions on 
(finite-dimensional but not necessarily locally finite) CAT(0) cube complexes.  
Examples include actions on Bass-Serre trees associated to finite 
graphs of groups, but this class also encompasses the large array of groups 
known to act on higher-dimensional CAT(0) cube complexes satisfying these 
conditions,  
e.g.~\cite{NibloReeves,CharneyDavis,Wise:small_can,OllivierWise,HagenPrzytycki,
HagenWise,LauerWise,Martin:Higman,WiseTubular,BergeronWise,Sageev97}.  

This paper highlights a new property of actions on CAT(0) cube complexes: if 
$G$ acts on the CAT(0) cube complex $\Psi$, we say that $G$ acts 
\emph{hyperplane-essentially} if for each hyperplane $H$ of $\Psi$, the 
stabiliser $\stabilizer_G(H)$ of $H$ acts essentially on the CAT(0) cube 
complex $H$.  (Recall that $G$ acts essentially on $\Psi$ if, for each halfspace 
in $\Psi$, any $G$--orbit contains points of that halfspace arbitrarily far 
from its bounding hyperplane.)  Work of Caprace-Sageev shows that, under 
reasonable conditions on the $G$--action, one can always pass to a convex 
subcomplex of $\Psi$ on which $G$ acts essentially, but simple examples (where 
$G=\integers$) show that passing to a convex subcomplex may never yield a 
hyperplane-essential action.  

There are numerous reasons to be interested in hyperplane-essentiality, which 
is a weak version of ``no free faces''.  For example, hyperplane-essentiality 
enables one to apply results guaranteeing that intersections of halfspaces in 
CAT(0) cube complexes contain hyperplanes (again, under mild conditions on the 
complex), strengthening the very useful~\cite[Lemma 5.2]{CapraceSageev}.  
Access to these lemmas has various useful consequences; for instance, 
hyperplane-essentiality is used in the forthcoming~\cite{HagenWilton} to 
generalise Guirardel's \emph{core} of a pair of splittings of a group $G$ to a 
core of multiple cubulations of $G$.  

Other consequences of hyperplane-essentiality arise from one of the typical 
ways in which it can fail.  First, note that if $\Psi$ is compact and $G$ acts 
on $\Psi$, then the action is essential if and only if $\Psi$ has no 
halfspaces, i.e. $\Psi$ is a single point.  Hence, if $\Psi$ is a CAT(0) cube 
complex with compact hyperplanes, then a $G$--action on $\Psi$ is 
hyperplane-essential if and only if $\Psi$ is a tree.

More generally, suppose that $G$ acts cocompactly on $\Psi$, but there is some 
hyperplane $H$ so that the action of $\stabilizer_G(H)$ on $H$ is not 
essential.  Then, since $\stabilizer_G(H)$ must act cocompactly on $H$, there 
exists a hyperplane $E$ that crosses $H$ and is \emph{extremal} in $H$, in the 
sense that, for some halfspace $E^+$ associated to $E$, the halfspace $E^+\cap 
H$ of $H$ lies entirely in the cubical neighbourhood of $E\cap H$.  This leaves 
the $G$--action on $\Psi$ open to the main technique introduced in this paper, 
\emph{panel collapse}, which is inspired by an idea in~\cite{touikan2015one}, 
in which certain square complexes are equivariantly collapsed to simpler ones.  
This procedure enables a $G$--equivariant deformation 
retraction from $\Psi$ to a lower--\emph{complexity} CAT(0) cube complex 
$\Psi_\bullet$.  Although $\Psi_\bullet$ need not be a subcomplex of $\Psi$, it 
is a $G$--invariant subspace with a natural cubical structure inherited from 
$\Psi$, whose hyperplanes extend to those of $\Psi$.  Specifically:

\begin{thmi}[Panel collapse, 
Corollary~\ref{cor:lower_complexity}]\label{thmi:panel_collapse}
Let $G$ act cocompactly and without inversions in hyperplanes on the CAT(0) 
cube complex $\Psi$.  Suppose that for some hyperplane $H$, the stabiliser of 
$H$ fails to act essentially on $H$ (i.e. $H$ contains a 
$\stabilizer_G(H)$--shallow halfspace; equivalently, a shallow halfspace). Then 
there is a CAT(0) cube 
complex 
$\Psi_\bullet$ such that:
\begin{enumerate}
 \item $\Psi_\bullet\subset\Psi$, and each hyperplane of $\Psi_\bullet$ is a 
component of a subspace of the form $K\cap\Psi_\bullet$, where $K$ is a 
hyperplane of $\Psi$. \label{it:cat0}
 
 \item $\Psi_\bullet$ is $G$--invariant and the action of $G$ on $\Psi_\bullet$ 
is cocompact.
 
 \item The action of $G$ on $\Psi_\bullet$ is without inversions in hyperplanes.
 
 \item $\Psi_\bullet$ has strictly lower \emph{complexity} than $\Psi$.
 
 \item Each $g\in G$ is hyperbolic on $\Psi$ if and only if it is hyperbolic on 
$\Psi_\bullet$.

 \item If $\Psi$ is locally finite, then so is $\Psi_\bullet$.
\end{enumerate}
\end{thmi}

The complexity of $\Psi$ is just the number of $G$--orbits of cubes of each 
dimension $>1$, taken in lexicographic order.  So, when the complexity 
vanishes, 
$\Psi$ is a tree.  A halfspace in a CAT(0) cube complex $Y$ is \emph{shallow} 
if it is contained in some finite neighbourhood of its bounding hyperplane, 
and, given a $G$--action on $Y$, a halfspace is \emph{$G$--shallow} if some 
$G$--orbit intersects the halfspace in a subset contained in a finite 
neighbourhood of the bounding hyperplane.

From Theorem~\ref{thmi:panel_collapse}, induction on complexity then shows 
that, if $G$ acts cocompactly on a CAT(0) cube complex $\Psi$, then \textbf{$G$ 
acts cocompactly and hyperplane-essentially on some other CAT(0) cube complex 
$\Omega\subset\Psi$}.  Moreover, one can then
pass to the Caprace-Sageev \emph{essential core}, as in~\cite[\S 3]{CapraceSageev}, to obtain a cocompact, essential, hyperplane-essential 
action of $G$ on a CAT(0) cube complex $\Omega$.  

In particular, if the hyperplanes of $\Psi$ 
were all compact, then the hyperplanes of $\Omega$ are all single points, i.e. 
$\Omega$ is a tree on which $G$ acts minimally.  In other words, we find a 
nontrivial splitting of $G$ as a finite graph of groups.

The main technical difficulty is that hyperplanes can intersect their 
$G$--translates (indeed, in our application to Stallings's theorem, discussed 
below, this is the whole source of the problem).  So, naive approaches 
involving collapsing free faces cannot work, and this is why our procedure 
gives a subspace which is not in general a subcomplex.

We now turn to consequences of Theorem~\ref{thmi:panel_collapse}.  

\subsection*{Stallings's theorem}\label{sec:intro_stallings}
Stallings's 1968 theorem on groups with more than one end is one of the
most significant results of geometric group theory:

\begin{thmi}[{\cite{Dunwoody-fg}} Stallings's theorem, modern 
formulation]\label{thm:Stallings}
  If $G$ is a finitely generated group, $X$ a Cayley graph
  corresponding to a finite generating set, and $K$ a compact subgraph such that
  $X\setminus K$ has a at least two distinct unbounded connected
  components, then $G$ acts nontrivially and with finite edge stabilizers on a 
tree $T$.
\end{thmi}

Theorem~\ref{thm:Stallings} is proved as Corollary~\ref{cor:stallings}
below. The essence of the proof is the fact that if $G$ acts on a
cocompactly on a CAT(0) cube complex with compact hyperplanes then,
successively applying Theorem \ref{thmi:panel_collapse}, we can make
$G$ act on a tree.

There are numerous proofs of Theorem~\ref{thm:Stallings}.  Our proof avoids 
certain combinatorial arguments (e.g. 
\cite{Kron,evangelidou2014cactus}) and analysis (see
\cite{kapovich2014energy}).  Ours is not the first proof of Stallings's theorem 
using CAT(0) cube complexes.  In fact, the ideas involved in the original proof 
anticipate CAT(0) cube complexes to some extent.  

Stallings's original proof (for the finitely presented case)
\cite{stallings1968torsion,stallings-gt3dm} precedes the development 
of Bass-Serre theory. Instead, Stallings
developed so-called \emph{bipolar structures}. Equipped with Bass-Serre theory, 
the problem boils down to dealing with a finite separating subset $K$ that 
intersects some of its $G$-translates.

Dunwoody, in \cite{dunwoody-1985}, while proving accessibility, gave a
beautiful geometric proof of Stallings's theorem in the almost
finitely presented case using the methods of patterns in polygonal
complexes. It is relatively easy to turn the separating set $K$ into a
\emph{track}. Using a minimality argument, he shows that it is
possible to cut and paste tracks until they become disjoint,
preserving finiteness and separation properties.

CAT(0) cube complexes became available via work of Gerasimov and 
Sageev~\cite{sageev1995ends,gerasimov1997semi}, employing the 
notion of \emph{codimension-1} subgroups.  (More generally, one can cubulate a 
\emph{wallspace}~\cite{nica2004cubulating,chatterji2005wall}.) Cube 
complexes are a very natural platform for addressing the intersecting cut-set 
problem. 

Niblo, in \cite{niblo2004geometric}, gave a proof of
Stallings's theorem using CAT(0) cube complexes, in fact proving something more 
general about codimension--$1$ subgroups (see Corollary \ref{cori:kropholler}.) His method is to cubulate, and 
then use the cube complex to get a 2-dimensional complex (the 
2-skeleton of the cube complex), on which he is then able to use a minimality 
argument for tracks to get disjoint cut-sets. So, Niblo's proof uses 
the CAT(0) cubical action as a way to get an action on a $2$--complex with a 
ready-made system of tracks, namely the traces of the hyperplanes on the 
$2$--skeleton.  

Our proof of Theorem~\ref{thm:Stallings}, based on panel collapse, is
fundamentally different.  Rather than performing surgery on Dunwoody
tracks in the $2$--skeleton of the cube complex, we collapse the
entire complex down to an essential tree with finite edge stabilizers.

\subsection*{Cube complexes associated to line 
patterns: an example when $G$ is not finitely generated}
Cashen and Macura, in \cite{cashen-macura}, prove a remarkable
rigidity theorem which states that to any free group equipped with a
rigid line pattern, there is a pattern preserving quasi-isometry to a
CAT(0) cube complex, the \emph{Cashen-Macura complex}, equipped with a
line pattern $(X,\mathcal L)$ so that any line pattern preserving
quasi-isometry between two free groups is conjugate to an isometry
between the Cashen-Macura complexes.

In their paper the authors ask whether the cut sets of the
decomposition used to construct $X$ can be chosen so that $X$ is a
tree. While not answering this question directly, we show how panel
collapse can bring $X$ to a tree. This
gives:

\begin{thmi}[{Theorem \ref{prop:cashen-macura-tree},
  c.f. \cite[Theorem 5.5]{cashen-macura}}]
  Let $F_i,\mathcal L_i,i=0,1$, be free groups equipped with a rigid
  line patterns. Then there are locally finite trees $T_i$ with 
  line patterns $\mathcal L_i$ and embeddings
  \begin{equation}\label{eqn:embedding}
    F_i\stackrel{\iota_i}{\hookrightarrow} \mathrm{Isom}(T_i)
  \end{equation}
  inducing cocompact isometric actions $F_i\circlearrowright T_i$,
  which in turn induce equivariant line pattern preserving
  quasi-isometries \[ \phi_i:F_i \to T_i.\] Furthermore for any line
  pattern preserving quasi-isometry $q:F_0\to F_1$ there is a line
  pattern preserving isometry $\alpha_q$ such that the following
  diagram of line pattern preserving quasi-isometries commutes up to
  bounded distance:\[
    \begin{tikzpicture}[scale=1.5]
      \node (TL1) at (-1,0) {($T_0,\mathcal L_0)$};
      \node (TL2) at (1,0) {($T_1,\mathcal L_1)$};
      \node (F1L) at (-1,-1) {($F_0,\mathcal L_0)$};
      \node (F2L) at (1,-1) {($F_1,\mathcal L_1)$};
      \draw[->] (F1L) --node[left]{$\phi_0$} (TL1);
      \draw[->] (F2L) --node[right]{$\phi_1$} (TL2);
      \draw[->] (F1L) --node[above]{$q$} (F2L);
      \draw[->] (TL1) --node[above]{$\alpha_q$} (TL2);
    \end{tikzpicture}.
  \]
  We note that the actions of $F_i$ on $T_i$ are free since the
  quasi-isometries $\phi_i:F_i \to T_i$ are equivariant.
\end{thmi}

This is proved in Section \ref{sec:cashen_macura}.  By passing to
trees, tree lattice methods can be brought to bear and the authors
suspect that this result could play an important role in the
description of quasi-isometric rigidity in the class of graphs of free
groups with cyclic edge groups.

This application illustrates that in Theorem
\ref{thmi:panel_collapse}, we are not requiring $G$ to be finitely
generated. The result holds even when dealing with an uncountable
totally disconnected locally compact group of isometries of some cube
complex $X$.

\subsection*{The Kropholler conjecture}\label{subsec:kropholleri}
Around 1988, Kropholler made the following conjecture: given a finitely generated group $G$ and a subgroup $H\leq G$, then 
the existence of a proper $H$--almost invariant subset $A$ of $G$ with $AH=A$ ensures that $G$ splits nontrivially over a 
subgroup commensurable with a subgroup of $H$.  This conjecture has been verified under various additional 
hypotheses~\cite{KarNiblo,Kropholler,DunwoodyRoller,niblo2004geometric}; a proof using a different approach to the one taken 
here appears in~\cite{dunwoody2017structure}, although this proof is believed, at the time of writing, to contain a 
gap\footnote{Personal communication from Martin Dunwoody and Alex Margolis.}.  Niblo and Sageev have observed that the 
Kropholler conjecture can 
be rephrased in terms of actions on CAT(0) cube complexes~\cite{NibloSageev}.  In cubical language, the conjecture states 
that if $G$ acts essentially on a CAT(0) cube complex with a single $G$--orbit of hyperplanes, and $H$ is a hyperplane 
stabilizer acting with a global fixed point on its hyperplane, then $G$ splits over a subgroup commensurable with a subgroup 
of $H$.  In Section~\ref{subsec:kropholler}, we prove this under the additional hypothesis that $G$ acts cocompactly on the 
CAT(0) cube complex in question (making no properness assumptions on either the action or the cube complex).  Specifically, 
we use panel collapse to obtain a very short proof of:

\begin{cori}[Corollary~\ref{cor:kropholler}]\label{cori:kropholler}
Let $G$ be a finitely generated group and $H\leq G$ a finitely
generated subgroup with $e(G,H)\geq2$. Let $\Psi$ be the dual cube
complex associated to the pair $(G,H)$, so that $\Psi$ has one
$G$--orbit of hyperplanes and each hyperplane stabiliser is a
conjugate of $H$.  Suppose furthermore that:
\begin{itemize}
 \item $G$ acts on $\Psi$ cocompactly;
 \item $H$ acts with a global fixed point on the associated hyperplane.
\end{itemize}
Then $G$ admits a nontrivial splitting over a subgroup commensurable with a subgroup of $H$.
\end{cori}

In the case where $G$ is word-hyperbolic and $H$ a quasiconvex
codimension--$1$ subgroup, this follows from work of
Sageev~\cite{Sageev97}, but the above does not rely on hyperbolicity,
only cocompactness. A similar statement, \cite[Theorem
B]{niblo2004geometric}, was proved by Niblo under the additional
hypothesis that hyperplanes are compact. In our case the
cocompactness hypothesis and  the fact that $H$ acts with a fixed
point imply that hyperplanes have bounded diameter but need not be
compact. We hope that a future generalisation of panel collapse will
enable a CAT(0) cubical proof of the Kropholler conjecture that does
not require a cocompact action on the cube complex.

\subsection*{Other examples and applications}\label{subsec:other_examples}
Theorem~\ref{thmi:panel_collapse} has other applications in group theory.  For 
example, 
Theorem~\ref{thmi:panel_collapse} is used in a forthcoming paper of 
Hagen-Wilton to study subsets of outer space and the free splitting 
complex fixed by finite subgroups of the outer automorphism group of a free 
group~\cite{HagenWilton}.  In that paper, the authors construct, given a group 
$G$ acting cocompactly, essentially, and hyperplane-essentially on finitely 
many CAT(0) cube complexes, an analogue of the Guirardel core for those 
actions.  Among other applications, this core can be used to give a new proof 
of the Nielsen realisation theorem for $\Out(F_n)$, which is done 
in~\cite{HagenWilton} using Theorem~\ref{thmi:panel_collapse} (without using 
Stallings's theorem).  Since the original proof relies on Stallings's 
theorem~\cite{Culler:nielsen}, it is nice to see that in fact Stallings's 
theorem and Nielsen realisation for $\Out(F_n)$ in fact follow independently 
from the same concrete cubical phenomenon.

Theorem~\ref{thmi:panel_collapse} thus expands the arena in which the 
techniques from~\cite{HagenWilton} will be applicable, by showing that the 
hyperplane-essentiality condition can always be arranged to hold.  This is of 
particular 
interest because of ongoing efforts to define a ``space of cubical 
actions'' for a given group.  Even for free groups, it's not completely clear 
what this space should be, i.e. which cubical actions count as points in this 
space.  However, since known arguments giving e.g. connectedness depend on the 
aforementioned version of the Guirardel core, one should certainly restrict 
attention to cubical actions where the hyperplane stabilisers act essentially 
on hyperplanes.  This is what we mean in asserting that 
Theorem~\ref{thmi:panel_collapse} is aimed at future applications.  We 
investigate this very slightly in the present paper, in 
Section~\ref{sec:antenna}, where we show that the exotic cubulations of free 
groups constructed by Wise in~\cite{Wise:recube} do not have essentially-acting 
hyperplane stabilisers, and are thus subject to panel collapse.

\subsection*{Outline of the paper} Throughout the paper, we assume
familiarity with basic concepts from the theory of group actions on
CAT(0) cube complexes; see e.g.~\cite{Sageev:pcmi,Wise:RR}.  In
Section~\ref{sec:ingredients}, we define \emph{panels} in a CAT(0)
cube complex and discuss the important notion of an \emph{extremal
  panel}.  In Section~\ref{sec:cube_def_ret}, we describe panel
collapse within a single cube, leading to the proof of
Theorem~\ref{thmi:panel_collapse} in
Section~\ref{sec:cube_def_ret_extremal}. Section~\ref{sec:examples}
contains applications.

\subsection*{Acknowledgements} We thank the organisers of the
conference \emph{Geometric and Asymptotic Group Theory with
  Applications 2017}, at which we first discussed the ideas in this
work.  MFH thanks Henry Wilton for discussions which piqued his
interest in mutilating CAT(0) cube complexes, and Vincent Guirardel
for discussion about possible future generalizations
of panel collapse. The authors are also grateful to Mladen Bestvina
and Brian Bowditch for bringing up ideas related to the collapse of
CAT(0) cube complexes, and to Michah Sageev for encouraging us to apply panel collapse to the Kropholler conjecture.  We 
also thank Alex Margolis for pointing out the reference~\cite{niblo2004geometric} and the anonymous referee for many helpful 
comments which improved the exposition and simplified certain proofs.  MFH was supported by Henry Wilton's EPSRC
grant EP/L026481/1 and by EPSRC grant EP/R042187/1.

\section{Ingredients}\label{sec:ingredients}

\subsection{Blocks and panels}\label{subsec:block_panel}
Throughout this section, $\Psi$ is an arbitrary CAT(0) cube complex and all 
hyperplanes, subcomplexes, etc. lie in $\Psi$.  Recall that a 
\emph{codimension--$n$ hyperplane} in $\Psi$ is the (necessarily nonempty) 
intersection of $n$ distinct, pairwise intersecting hyperplanes. 

We will endow $\Psi$ with the usual CAT(0) metric $\dist_2$ (which we refer to 
as the \emph{$\ell_2$--metric}) in which all cubes are 
Euclidean unit cubes.  
There is also an \emph{$\ell_1$--metric} $\dist_1$, which extends the usual 
graph-metric on the $1$--skeleton.  Recall from~\cite{Haglund:semisimple} that 
a 
subcomplex is convex with respect to $\dist_2$ if and only if it is convex with 
respect to $\dist_1$ (equivalently, it is equal to the intersection of the 
combinatorial halfspaces containing it), but we will work with $\ell_2$--convex 
subspaces that are not subcomplexes.

Recall that convex subcomplexes of CAT(0) cube complexes have the \emph{Helly 
property}: if $C_1,\ldots,C_n$ are pairwise-intersecting convex subcomplexes, 
then $\bigcap_{i=1}^nC_i\neq\emptyset$.

Given $\mathcal A\subset\Psi$, the \emph{convex hull} of $\mathcal A$ is the 
intersection of all convex subcomplexes containing $\mathcal A$.  In 
particular, if $\mathcal A$ is a set of $0$--cubes in some cube $c$, then the 
convex hull of $\mathcal A$ is just the smallest subcube of $c$ containing all 
the $0$--cubes in $\mathcal A$.

The word ``hyperplane'', unless stated otherwise, means 
``codimension--$1$ hyperplane''.

\begin{defn}[Codimension--$n$ carrier]\label{defn:carrier}
  The \emph{carrier} $N(H)$ of a codimension $n$ hyperplane $H$ is the
  union of closed cubes containing $H$. We have, by results 
in~\cite{sageev1995ends},\[
    N(H) = H \times \interval{[}{-\half,\half}{]}^n.
  \]
\end{defn}

If $H$ is a codimension--$n$ hyperplane, then $H=\bigcap_{i=1}^nE_i$, where each 
$E_i$ is a hyperplane, and $N(H)=\bigcap_{i=1}^nN(E_i)$, where each $N(E_i)\cong 
E_i\times\interval{[}{-\half,\half}{]}$.

\begin{defn}[Blocks, panels, parallel]\label{defn:panel}
  Let $H$ be a codimension 2 hyperplane.  Then the carrier
  $\bur H = H \times \interval{[}{-\half,\half}{]}^2 $ is called a
  \emph{block}. The four closed
  faces\[ H \times \interval{[}{-\half,\half}{]} \times
    \interval{\{}{\pm \frac 1 2}{\}}, H \times \interval{\{}{\pm \frac
      1 2}{\}} \times \interval{[}{-\half,\half}{]}
  \] are called \emph{panels}. Panels are \emph{parallel} if they have
  empty intersection.  Equivalently, panels are parallel if they intersect 
exactly the same hyperplanes.  (More generally, two convex subcomplexes are 
said to be parallel if they intersect the same hyperplanes; parallel 
subcomplexes are isomorphic, see~\cite{BHS:I}.)  Since carriers of 
codimension--$1$ hyperplanes 
are convex subcomplexes, and each carrier of a codimension--$n$ hyperplane is 
the intersection of carriers of $n$ mutually intersecting codimension--$1$ 
hyperplanes, codimension--$n$ carriers are also convex.  From this and the 
product structure, it follows that panels are convex (justifying the term 
``parallel'').

If $P=H \times \interval{[}{-\half,\half}{]} \times
    \interval{\{}{\frac 1 2}{\}}\subset\bur H$ is a panel, and $e$ is a 
$1$--cube of $P$, we say that $e$ 
is \emph{internal (to $P$)} if it is not contained in $H \times 
\interval{\{}{\pm\half}{\}} \times    \interval{\{}{\frac 1 2}{\}}$.  The 
\emph{interior} of $\bur{H}$ is $$H\times\interval{(}{-\half,\half}{)}^2,$$ and 
the \emph{inside} of $P=H \times \interval{[}{-\half,\half}{]} \times
    \interval{\{}{\frac 1 2}{\}}$ is $\interior{P}=H \times 
\interval{(}{-\half,\half}{)} \times
    \interval{\{}{\frac 1 2}{\}}.$
\end{defn}

We emphasise that a subcomplex can be a panel in multiple ways, with different 
inside, according to which hyperplane is playing the role of $H$.  When we talk 
about a specific panel, it is with a fixed choice of $H$ in mind.  The choice of 
$H$ determines the inside of the panel uniquely.

\subsection{Extremal panels}\label{sec:extremal}
\begin{defn}[Extremal side, extremal panel]\label{defn:extremal}
  Let $H,E \subset \Psi$ be distinct intersecting hyperplanes. $H\cap E$ is 
\emph{extremal}
  in $H$ if the carrier $N_H(H\cap E)$ of $E\cap H$ in $H$ contains
  one of the halfspaces $H\cap \vec E$ of $H$ (here $\vec E$ is one of the 
components of the complement of $E$ in $\Psi$). The union of cubes of
  $H$ contained in $H \cap \vec E$ is called the \emph{extremal side}
  $H_E^+$ of $H$ with respect to $E$. The \emph{extremal panel}
  $P(H^+_E)$ is the minimal subcomplex of $\Psi$ containing all cubes
  intersecting $H^+_E$.
\end{defn}

If $H\cap E$ is extremal in $H$, then $P(H^+_E)$ is a panel of
$\bur{H\cap E}$ in the sense of Definition~\ref{defn:panel}.  We say
$H$ \emph{abuts} $P(H^+_E)$ and $E$ \emph{extremalises} $P(H^+_E)$.
Given an extremal panel $P$ each maximal cube $c$ of $P$ is contained
in a unique maximal cube of $\Psi$ and $c$ has codimension--$1$ in
that cube.

\begin{lem}[Finding extremal hyperplanes]\label{lem:compact_extremal}
Let $H$ be a bounded CAT(0) cube complex.  Then either $H$ is a single 
$0$--cube, or $H$ contains an extremal hyperplane.
\end{lem}

\begin{proof}
Let $x\in H$ be a $0$--cube and let $K$ be a hyperplane so that 
$d=\dist(x,N(K))$ is maximal as $K$ varies over all hyperplanes; such a $K$ 
exists because $d\in\naturals$ and $H$ is bounded.  By construction, $K$ cannot 
separate a hyperplane of $H$ from $x$, and thus $K$ is extremal in $H$.
\end{proof}

Lemma~\ref{lem:compact_extremal} immediately yields:

\begin{cor}\label{cor:compact_hyp}
Let $\Psi$ be a CAT(0) cube complex with a compact hyperplane that is not a single point.  Then $\Psi$ contains an extremal panel.
\end{cor}

\subsection{The no facing panels property}\label{sec:BIP}

\begin{defn}[No facing panels property]\label{defn:block_intersection_condition}
We say that $\Psi$ (respectively, $\Psi$ and a particular collection $\mathcal 
P$ of extremal panels) satisfies the \emph{no facing panels property} if the 
following holds.  Let $\bur{E\cap H}$ and $\bur{E^1\cap H^1}$ be distinct 
blocks so that $P=P(H^+_E)$ and $P^1((H^1)^+_{E^1})$ are extremal panels 
(respectively, in $\mathcal P$).  Suppose that $P\cap P^1=\emptyset$.  Then 
$\bur{E\cap H}$ and $\bur{E^1\cap H^1}$ do not have a common maximal cube.  
\end{defn}

\begin{rem}
 The maximal cubes of $\bur{E\cap H}$ have the form $f\times 
[-\frac12,\frac12]^2$, where $f$ is a maximal cube of $E\cap H$.  Hence, if 
$\bur{E\cap H},\bur{E^1\cap H^1}$ have a common maximal cube, it has the form 
$f\times  [-\frac12,\frac12]^2=f^1\times[-\frac12,\frac12]^2$, where $f\subset 
E\cap H,f^1\subset E^1\cap H^1$.  
\end{rem}

Recall that $G\leq\Aut(\Psi)$ acts \emph{without inversions} if, for all $g\in 
G$ and all hyperplanes $H$, if $gH=H$ then $g$ preserves both of the 
complementary components of $H$.

\begin{lem}\label{lem:block_intersection_condition}
Let $G\leq\Aut(\Psi)$ act without inversions across hyperplanes, let $P$ be an 
extremal panel, and let $\mathcal P=G\cdot P$.  Then $\mathcal P$ has the 
no facing panels property.
\end{lem}

\begin{proof}
Suppose that $\bur{E\cap H}$ and $\bur{E^1\cap H^1}$ have a common maximal cube 
$c$.  Let $\bar P,\bar P^1$ be parallel copies of $P,P^1$ in $\bur{E\cap 
H},\bur{E^1\cap H^1}$ respectively, so that $\bar P,P$ are parallel and 
separated by $E$ (and no other hyperplanes) and $\bar P^1,P^1$ are parallel and 
separated only by $E^1$.

If $P\cap {P^1}=\emptyset$, then $\bar P\cap P^1$ contains a codimension--$1$ 
face $\bar c'$ of $c$, whose opposite in $c$, denoted $c'$, lies in $P$.  The 
unique hyperplane separating $\bar c',c'$ is $E$.

By hypothesis, $P^1=gP$ for some $g\in G$, and extremality of $P^1$ implies 
that 
$g\bur{E\cap H}=\bur{E^1\cap H^1}$.  Hence either $gE=E^1$ and $gH=H^1$ or 
$gE=H^1$ and $gH=E^1$.  But $gH$ intersects $P^1$ since $H$ intersects $P$, 
while $E$ is disjoint from $P$, so $gE$ is disjoint from $P^1$.  Hence $gE=E^1$ 
and $gH=H^1$.  

Now, since $\bur{E^1\cap H^1}$ contains $c$ as a maximal cube, $E$ intersects 
$\bur{E^1\cap H^1}$.  On the other hand, $\bar c'$ is parallel to a subcomplex 
of $E$ (since 
$\bar c'$ is a cube of $\bar P$) while $\bar c'$ is also parallel into $E^1$ 
(since $\bar c'$ is a cube of $P^1$).  But, by the definition of an 
$H^1$--panel 
in $\bur{E^1\cap H^1}$, each maximal cube of $P^1$ (i.e. codimension--$1$ face 
of the maximal cube of the block) is parallel to a subcomplex 
of a unique hyperplane of 
$\bur{E^1\cap H^1}$, namely $E^1$.  Hence $E=E^1$, so $gE=E$.  But since $P$ 
and 
$gP=P^1$ are separated by $E$, the element $g$ acts as an inversion across $E$, 
a contradiction.
\end{proof}

\subsection{Motivating examples}
The original motivation for our main construction was the following:
knowing that compact CAT(0) cube complexes are collapsible to points,
how can one modify an ambient cube complex with compact hyperplanes
so as to realize a collapse of the hyperplanes, while keeping the ambient space a CAT(0) cube complex?

First consider a single $3$--cube $c$. If we were to collapse a free
face, i.e. to delete an open square in the boundary and the open
$3$--cube within $c$, we would have a union of 5 squares giving an
``open box'', which is not CAT(0). To preserve CAT(0)ness, we need to ``collapse
an entire panel'', as follows.

Let $H,E$ be two distinct hyperplanes in $c$ and pick a halfspace
$\vec E\cap H$. Then the intersection of $\vec E\cap H$ with the
$1$--skeleton of $c$ is a pair of $1$--cube midpoints. These open
$1$--cubes will be called internal and the maximal codimension--$1$
face containing all these is a panel. $H$ is the abutting hyperplane
and $E$ is the extremalising hyperplane. We obtain a \emph{deletion}
$\mathcal D(c)$ by removing all open cubes containing the internal $1$-cubes. This amounts to applying 
Theorem~\ref{thm:pasting} once,
in the case where $\Psi$ is the single cube $c$ and $\mathcal P$ consists of a single panel $P$.  Passage from $c$ to $\mathcal{D}(c)$ can be realised by a strong deformation retraction.  
See Figure~\ref{fig:single_taco}.

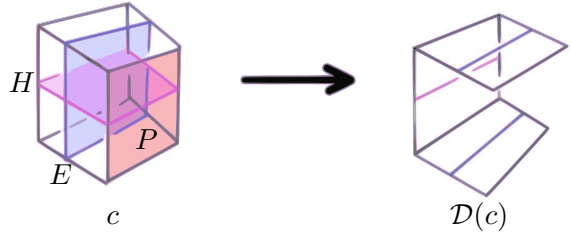
\begin{figure}[h]
\begin{overpic}[width=0.5\textwidth]{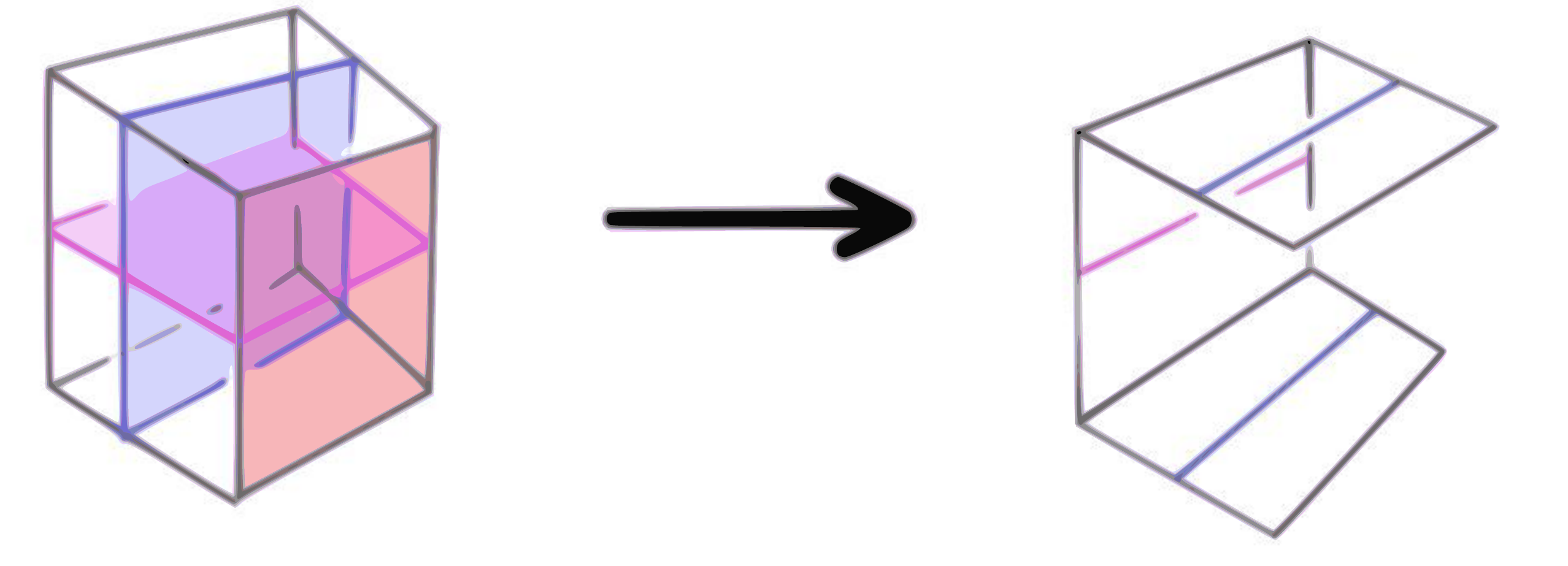}
\put(15,-3){$c$}
\put(75,-3){$\mathcal D(c)$}
\put(-2,20){$H$}
\put(5,4){$E$}
\put(20,10){$P$}
\end{overpic}
\caption{Collapsing the panel $P$ of the cube $c$ to obtain $\mathcal D(c)$.  The extremalising hyperplane is $E$ and the abutting hyperplane is $H$.}\label{fig:single_taco}
\end{figure}

The resulting cube complex is again CAT(0). By repeatedly choosing abutting and externalizing hyperplanes, and applying Theorem~\ref{thm:pasting} several times in succession, we
can collapse the cube down to a tree, as in Figure~\ref{fig:taco_sequence}. In this tree, the edge-midpoints
arise as intersections of this tree sitting inside $c$ and the
original hyperplanes of $c$.

\begin{figure}[h]
\includegraphics[width=0.65\textwidth]{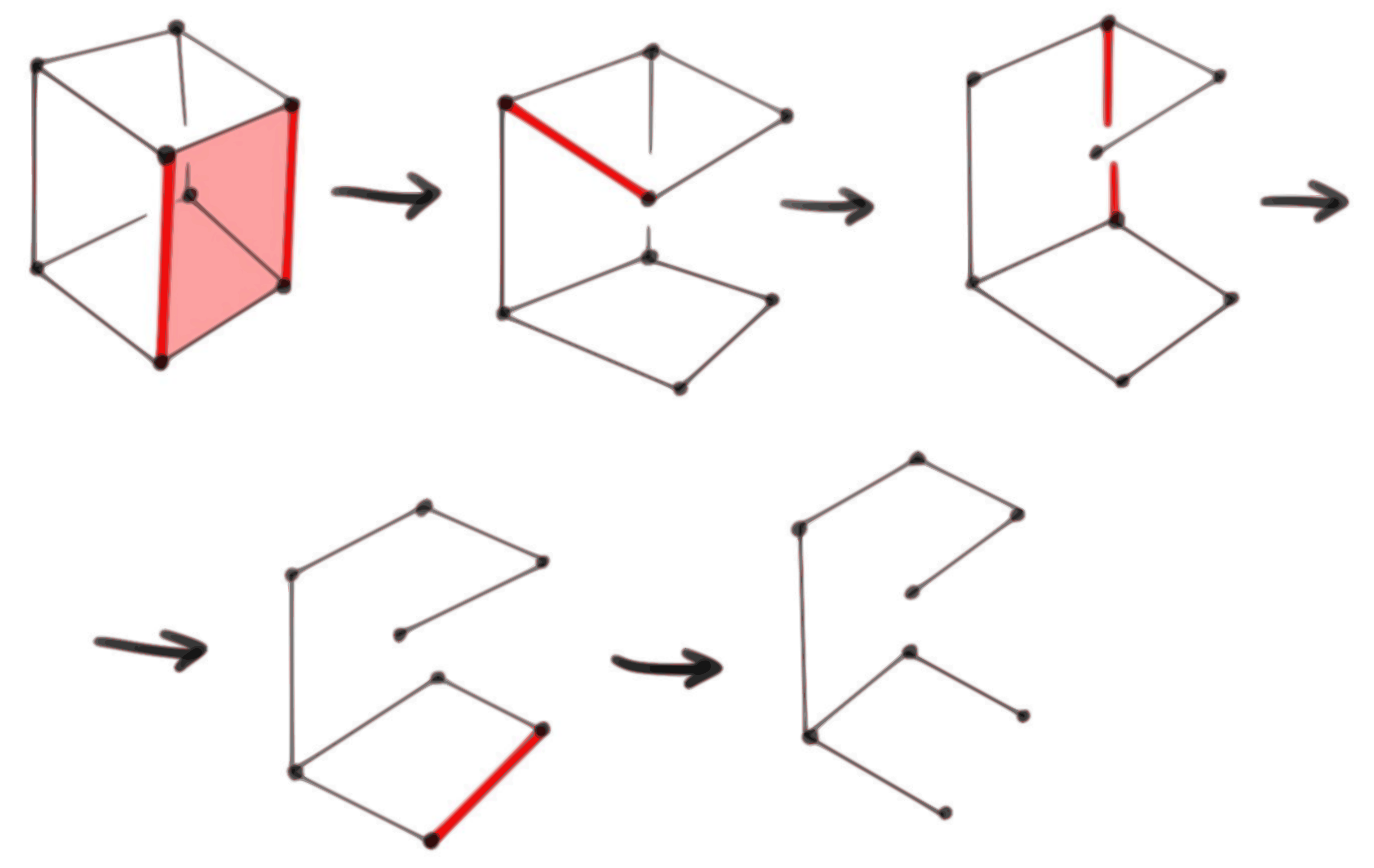}
\caption{Reducing $c$ to a tree via a sequence of panel collapses.  In this picture, at each stage a new panel (shaded) is chosen in the current subcomplex of $c$.  In all but the first step, the panel in question is a $1$--cube.  At the final stage, we have a subcomplex of $c$, so the midpoints of edges are contained in hyperplanes of $c$.  The bold $1$--cubes are dual to the abutting hyperplanes of the panels being collapsed.}\label{fig:taco_sequence}
\end{figure}

Now suppose we wanted to collapse many panels simultaneously (instead of by applying Theorem~\ref{thm:pasting} to a single panel, choosing a panel in the resulting complex, and iterating).  The most serious conflict between panels that could arise involves two panels intersecting a common cube in opposite faces, but this is ruled out by the requirement that we draw our panels from a collection with the no facing panels property.  The other issue involves two panels that intersect a common cube, which cannot be so easily hypothesised away.  This brings us to our second basic example, the situation where $\Phi$ is an infinite CAT(0) square complex with compact hyperplanes that admits a cocompact action by some group $G$. 

In this case, hyperplanes are trees
and panels are free faces of squares (corresponding to leaves of the
hyperplane-trees).  Although we will require that the action does not
invert hyperplanes (not a real restriction because one can subdivide), it may very well be that some element of $G$ maps
a square $c$ to itself, so that $c$ contains two distinct panels $P$ and
$gP$. By the no facing panels property, these panels must touch at some
corner of $c$. If we wish to $G$-equivariantly modify $\Phi$, then we
must collapse both $P$ and $gP$. The resulting $\mathcal D(c)$ obtained
by deleting all the open cubes of $c$ whose interiors lie in the
interior of $P$ or $gP$ is disconnected (one of the components is a $0$--cube). However, if $h(c)$ is the
$0$--cube of $c$ in $\mathcal D(c)$ diagonally opposite the isolated $0$--cube, then it can be connected to its
opposite $\bar h(c)=w$ by a diagonal $\mathcal S(\bar h(c))$.  The cubes $h(c),\bar h(c)$ are the \emph{persistent} and \emph{salient} cubes defined below (in general, they need not be $0$--cubes).  The resulting subspace $\mathcal F(c)$ replaces $c$ in the new complex, and we can make these replacements equivariantly.  See Figure~\ref{fig:double_taco}.

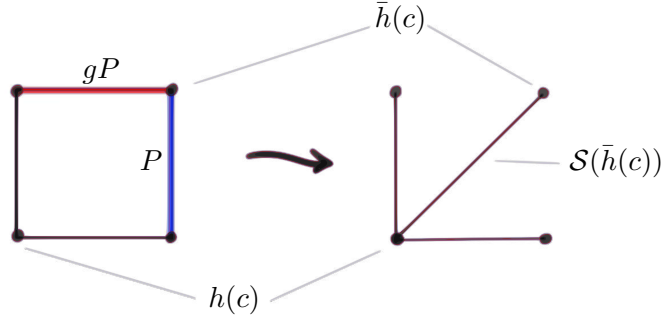
\begin{figure}[h]
\begin{overpic}[width=0.5\textwidth]{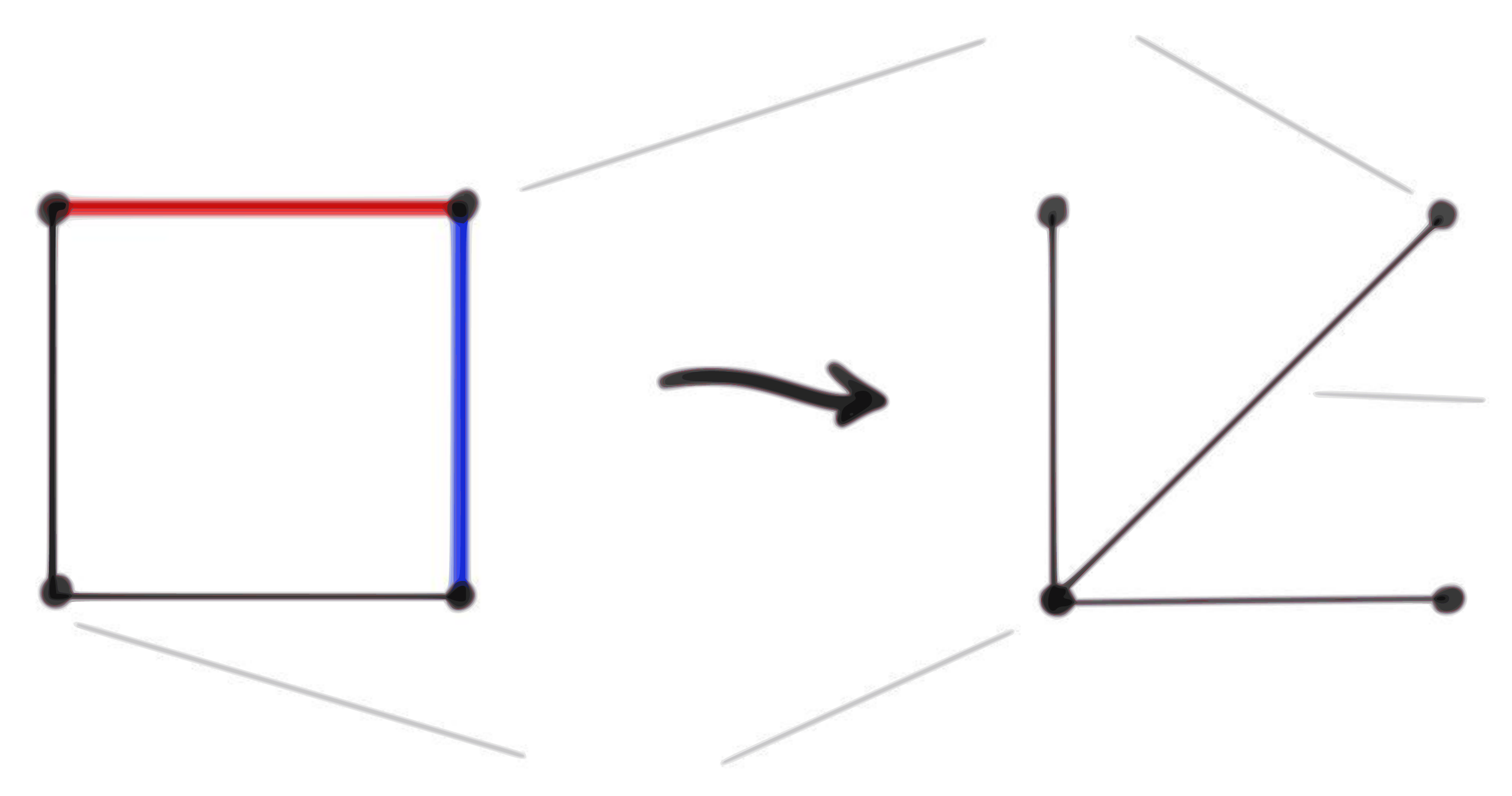}
\put(66,50){$\bar h(c)$}
\put(37,1){$h(c)$}
\put(100,25){$\mathcal S(\bar h(c))$}
\put(25,25){$P$}
\put(15,41){$gP$}
\end{overpic}
\caption{Passing from $c$ to $\mathcal F(c)$ when the given collection of panels contains two that conflict.  Note that 
$\mathcal F(c)$ is no longer a subcomplex of $c$, but it is a CAT(0) cube complex, embedded in $c$, whose $0$--skeleton is 
the same as that of $c$.}\label{fig:double_taco}
\end{figure}

The panel collapse construction, described formally in the next section, just generalises these procedures.

\section{Panel collapse for single cubes}\label{sec:cube_def_ret}
Let $\Psi$ be a finite-dimensional CAT(0) cube complex and let
$\mathcal P$ be a collection of extremal panels with the no facing
panels property.  Recall that the inside of $P$ is the union of open
cubes of $P$ that intersect the abutting hyperplane.  When we refer to
the \emph{interior} of a cube $[-\frac12,\frac12]^n$, we just mean
$(-\frac12,\frac12)^n$. If $c$ is a (closed) cube of $\Psi$ and
$P\in\mathcal P$, then $P\cap c$ is either $\emptyset$ or a sub-cube
of $c$, by convexity of $P$.

The first goal of this section is to define the \emph{fundament}
$\mathcal F(c)$ of a cube $c$, which is a subspace obtained from $c$
given the intersections $\{c\cap P\mid P \in \mathcal P \}$, which are
panels of $c$ that must be collapsed. This construction of
$\mathcal F(c)$ must be compatible with the induced collapses and
fundaments $\mathcal F(c')$ of the subcubes $c'$ of $c$. The second
goal of this section is to prove the existence of a strong deformation
retraction of $c$ onto $\mathcal F(c)$, compatible with deformation
retractions on the codimension--$1$ faces, and to show that
$\mathcal F(c)$ is again a CAT(0) cube complex.

\begin{lem}\label{lem:how_panels_intersect_cubes}
Let $c$ be a cube, let $P\in\mathcal P$, and let $e$ be a $1$--cube of $c$ that 
is internal to $P$.  Then $c$ has a codimension--$1$ face $f$ so that every 
$1$--cube of $f$ that is parallel to $e$ is internal to $P$.  
\end{lem}

In other words, if $c$ is a cube and $P\in\mathcal P$, then exactly one of the 
following holds:
\begin{itemize}
     \item $c\cap P=\emptyset$;
     \item the interior of $c$ is contained in the inside of $P$;
     \item $c\cap P$ is a codimension--$1$ face $f$ of $c$, and moreover, if 
$e$ is an edge of $f$ dual to the abutting hyperplane of $P$, then any edge of 
$f$ parallel to $e$ is in the inside of $P$.
\end{itemize}

\begin{proof}
Let $E,H$ be hyperplanes so that $\bur P=\bur{E\cap 
H}$, with $H$ dual to the internal $1$--cubes of $P$, so $e$ is dual to $H$.  
By extremality of $P$, we must have that $c\subset\bur P$.  If the interior of 
$c$ is contained in the inside of $P$, we're done.  Otherwise, $E\cap 
c\neq\emptyset$, and $c$ has two codimension--$1$ faces, $f,f'$, that are 
separated by $E$, with $f\subset P$.  Thus every $1$--cube of $f$ dual to $H$ 
(i.e. parallel to $e$) is internal to $P$, as required.
\end{proof}

The next lemma is immediate from the definition of a panel:

\begin{lem}\label{lem:kill_whole_class}
Let $c$ be a cube of $\Psi$.  If the interior of $c$ lies in the inside of 
some $P\in\mathcal P$, then there is some $1$--cube $e$ of $c$ such that $e'$ is 
internal for $P$ for all $1$--cubes $e'$ of $c$ parallel to $e$.

Conversely, if $c$ is a cube, and for some $1$--cube $e$ of $c$, every 
$1$--cube parallel to $e$ is internal to $P$, then the interior of $c$ is 
contained in the inside of $P$.
\end{lem}

From Lemma~\ref{lem:kill_whole_class} and convexity of panels, we obtain:

\begin{lem}\label{lem:convex_panel_consequence}
Let $c$ be a cube and let $P\in\mathcal P$.  Let $e,e'$ be $1$--cubes of $c$ 
that are internal to $P$.  Suppose that $e,e'$ do not lie in a common proper 
sub-cube of $c$.  Then the interior of $c$ is contained in the inside of $P$.

Hence, if the interior of $c$ is not contained in the inside of $P$, then all 
of the $1$--cubes of $c$ internal to $P$ are contained in a common 
codimension--$1$ face of $c$.
\end{lem}

\begin{defn}[Internal, external, 
completely external]\label{defn:internal_external}
A subcube $c$ of $\Psi$ is \emph{internal} [resp., internal to $P\in\mathcal P$] 
if its interior lies in the inside of 
some panel in $\mathcal P$ [resp., $P$].  Otherwise, $c$ is \emph{external}.  
If $c$ 
contains no $1$--cube that is internal, then $c$ is 
\emph{completely external}.  Completely external cubes are external, but 
external cubes are not in general completely external.  Specifically, $c$ is 
external but not completely external exactly when $c$ has a codimension--$1$ 
face that is internal.  (The notions of 
externality and complete externality coincide for $1$--cubes, and every 
$0$--cube is completely external.)
\end{defn}

The following two lemmas will be important later in the section, but we state them here since they are just 
about the definition of an external cube.

\begin{lem}\label{claim:opposite}
Let $f$ be a cube which is the product of external cubes.  Then $f$ is external.
\end{lem}

\begin{proof}
Write $f=f'\times e$, where $f',e$ are external cubes.

If $f$ is not external, then there exists a $1$--cube $e'$ such that every $1$--cube of $f$ parallel to $e'$ 
is internal.  
Since $e$ is external, $e'$ cannot be parallel to $e$.  Hence $e'$ is parallel to a $1$--cube of $f'$.  But 
since $f'$ is 
external, every parallelism class of $1$--cubes represented in $f'$ has an external representative in $f'$.  
This proves 
that $f$ is external.
\end{proof}

\begin{lem}\label{lem:external_cube_intersections}
  The intersection $c \cap c'$ of two external cubes $c,c'$ is
  external.
\end{lem}
\begin{proof}
  If $c' \subset c$, this is obviously true. If $c\cap c'$ is a
  0--cube, then this holds since all 0--cubes are external. Suppose
  now towards a contradiction, that $c \cap c'$ is an internal cube of
  dimension at least 1. Let $P\in \mathcal P$ be a panel containing
  the interior of $c \cap c'$. Let $H$ be the abutting hyperplane and
  let $E$ be the extremaliser of $P$. Since $c \cap c'$ is internal to
  $P$ it cannot contain a 1--cube dual to $E$. The hyperplanes $E$ and
  $H$ cross $c$, and by extremality $c'$ must lie in the carrier
  $\mathcal N(E)$. Since $c'$ is not contained in $P$, then
  $c' \cap P$ is a codimension--1 face of $c'$. Hence every 0--cube of
  $c'$ is contained in a 1--cube dual to $E$. In particular, since
  $c\cap c'$ contains a 0--cube $v$, the 1--cube $e$ containing $v$
  and dual to $E$ lies in $c \cap c'$, contradicting that $c \cap c'$
  is internal.
\end{proof}

\begin{rem}\label{rem:codim_1_external}
In several places, we will use the following consequence of Lemma~\ref{lem:how_panels_intersect_cubes}.  Let $e,e'$ be 
$1$--cubes of a cube $c$ that are parallel (i.e. dual to the same hyperplane $D$).  Suppose that every hyperplane of $c$ 
other than $D$ separates $e,e'$.  This is equivalent to the assertion that for each codimension--$1$ face $c'$ of $c$, at 
most one of $e,e'$ is contained in $c'$.  In this situation, if $e,e'$ are both external, then every $1$--cube $e''$ dual to 
$D$ must also be external.  Indeed, suppose that $e''$ is an internal $1$--cube dual to $D$.  Then there is a panel 
$P\in\mathcal P$ such that $e''$ is internal to $P$ and $e''$ lies in some codimension--$1$ face $f$ such that every 
$1$--cube of $f$ parallel to $e''$ is internal.  Now, either $e$ or $e''$ must lie in $f$, for otherwise both $e,e''$ must 
be contained in the codimension--$1$ face parallel to, and disjoint from, $f$.  But this is impossible, and hence the 
externality of $e$ or $e'$ is contradicted by the internality of $e''$.
\end{rem}

So far we have not used that $\mathcal P$ has the no facing panels 
property; now it is crucial.

\begin{lem}\label{lem:no_facing}
Let $c$ be a cube of $\Psi$.  Let $f,f'$ be a pair of parallel codimension--$1$ 
faces of $c$ so that there are panels $P,P'\in\mathcal P$ with the interiors of 
$f,f'$ respectively contained in $P,P'$.  Then $c$ is internal.
\end{lem}

\begin{proof}
Let $m$ be a maximal cube containing $c$.  Let $H$ be the hyperplane that 
crosses $c$ (and thus $m$) and separates $f,f'$.  Let $n,n'$ be the 
codimension--$1$ faces of $m$ that respectively contain $f,f'$ and are 
separated by $H$.  

If $P=P'$, then we are done, by Lemma~\ref{lem:convex_panel_consequence}.  
Hence suppose that $P\neq P'$ and $f'$ is not internal to $P$ and $f$ is not 
internal to $P'$.  Then, since $P,P'$ must intersect $m$ in codimension--$1$ 
faces, we have that the interior of $n$ lies in the inside of $P$, and the 
interior of $n'$ lies in the inside of $P'$, and this contradicts the no facing 
panels property.
\end{proof}

\begin{lem}[External cubes have persistent corners]\label{lem:cube_intersect}
Let $c$ be an external cube of $\Psi$. Then there exists a $0$--cube $v$ of 
$c$, and distinct $1$--cubes $e_1,\ldots,e_{\dimension c}$, contained in $c$ and 
incident to $v$, so that each $e_i$ is external.
\end{lem}

A 0--cube $v$ as in the lemma is a \emph{persistent corner} of $c$.

\begin{proof}
For each $P\in\mathcal P$, Lemma~\ref{lem:convex_panel_consequence} implies 
that $P\cap c$ is either empty or confined to a codimension--$1$ face $f_P$ of 
$c$.  Identifying $c$ with $f_P\times[-\half,\half]$ and $f_P$ with 
$f_P\times\{\half\}$, let $b_P=f_P\times\{-\half\}$, which is a 
codimension--$1$ face of $c$.  For each $P,P'$ as above, 
Lemma~\ref{lem:no_facing} shows 
that $b_P\cap b_{P'}\neq\emptyset$, so by the Helly property for convex 
subcomplexes, $\bigcap_Pb_P$ is  a nonempty sub-cube $c'$ of $c$ (the 
intersection 
is taken over all panels in $\mathcal P$ contributing an internal $1$--cube 
to $c$).  

Let $v$ be a $0$--cube of $c'$.  By construction, any $1$--cube $e$ of $c'$ 
incident to $v$ is external.  Next, suppose that $e$ is a $1$--cube of $c$ 
incident to $v$ that joins $c'$ to a $1$--cube not in $c'$.  If $e$ is internal 
to some panel $P$, then $e$ is separated from $b_P$ by a hyperplane of $c$.  
Since $v\in e$, we have $v\not\in b_P$, a contradiction.  Hence $e$ is 
external.
\end{proof}

\begin{defn}[Persistent subcube, salient subcube]\label{defn:persistent_subcube}
Let $c$ be an external cube of $\Psi$.  Lemma~\ref{lem:cube_intersect} implies 
that $c$ has at least one persistent corner.  Let $h(c)$ be the ($\ell_1$) convex hull of 
the persistent corners of $c$, i.e. the smallest sub-cube of $c$ containing all 
of the persistent corners.  We call $h(c)$ the \emph{persistent subcube} of 
$c$.  Let $\bar h(c)$ be the subcube of $c$ with the following properties:
\begin{itemize}
     \item $\bar h(c)$ is parallel to $h(c)$, i.e. $\bar h(c)$ and $h(c)$ 
intersect exactly the same hyperplanes;
\item the hyperplane $H$ of $c$ separates $h(c)$ from $\bar h(c)$ if and only 
if $H$ does not cross $h(c)$ and $H$ is dual to some internal $1$--cube of $c$.
\end{itemize}
We call $\bar h(c)$ the \emph{salient subcube} of $c$.  Note that $h(c)$ and 
$\bar 
h(c)$ are either equal or disjoint, and have the same dimension.  Also, if $c'$ 
is a sub-cube of $c$, and $h(c)\cap c'\neq\emptyset$, then $c'$ contains a 
persistent corner of $c$, by minimality of $h(c)$.
\end{defn}

\begin{lem}\label{lema:all_persistent}
Every $0$--cube of $h(c)$ is persistent.
\end{lem}

\begin{proof}
Adopt the notation of Lemma~\ref{lem:cube_intersect}.  So, $c$ has a nonempty subcube $c'$ which is the intersection 
$\bigcap_{P\in\mathcal P}b_P$, where for each panel $P\in\mathcal P$, $b_P$ is the codimension--$1$ face opposite $P\cap c$. 
 The proof of Lemma~\ref{lem:cube_intersect} shows that every $0$--cube of $c'$ is a persistent corner of $c$.  
 
 In particular, $c'\subseteq h(c)$.  Now suppose that $v$ is a $0$--cube of $c$.  Suppose that $v\in P\cap c$ for 
some $P\in\mathcal P$.  Then at least one $1$--cube of $P\cap c$ incident to $v$ is internal, so $v$ is not a persistent 
corner of $c$.  Hence every persistent corner of $c$ lies in $c'$, so $h(c)\subseteq c'$.  Thus $h(c)=c'$.  Since every 
$0$--cube of $c'$ is persistent, the lemma follows.
\end{proof}

\begin{rem}[Equivalent characterisation of persistent subcubes]\label{rem:equiv_char}
Given an external cube $c$ of $\Psi$, we can equivalently characterise $h(c)$ and $\bar h(c)$ as follows.  As 
in Lemma~\ref{lem:cube_intersect}, for each panel $P\in\mathcal P$ intersecting $c$, let $f_P$ be the 
codimension--$1$ face in which $P$ intersects $c$ and let $b_P$ be the opposite face.  Then 
$h(c)=\bigcap_Pb_P$.  
\end{rem}

\subsection{Building the subspace $\mathcal F(c)$}
Let $c$ be a (closed) cube of $\Psi$.  We define a subspace $\mathcal F(c)$ of 
$c$ as follows.  First, let $\mathcal D(c)$ be the union of all completely 
external cubes of $c$.  In other words, $\mathcal D(c)$ is obtained from $c$ by 
deleting exactly those open sub-cubes whose closures contain internal $1$--cubes.

Observe that $\mathcal D(c)\cap 
c'=\mathcal D(c')$ for all sub-cubes $c'$ of $c$.

Before defining $\mathcal F(c)$, we need the following important fact about $\mathcal D(c)$:

\begin{lem}\label{lem:connected_hereditary}
If $c$ is an external cube and $c'$ an external sub-cube of $c$, then 
$\mathcal D(c')$ is connected provided $\mathcal D(c)$ is 
connected.
\end{lem}

\begin{proof}
Let $d=\dimension c$ and $d'=\dimension c'$.  If $d=d'$, then $c=c'$ and the lemma is immediate.  If $d'=0$, 
then $\mathcal 
D(c')=c'$, and again the lemma holds.  Hence assume that $0<d'<d$.

Let $\pi:c\to c'$ be the canonical projection, where $c$ is regarded as $[-\frac12,\frac12]^{d}$ and 
$c'$ is regarded as $[-\frac12,\frac12]^{d'}$.  We first establish some auxiliary claims, the first of which is 
the special 
case of the lemma where $d-d'=1$.  

\begin{claim}\label{claim:codim_1_connected}
If $c$ is external and $c'$ is a codimension--$1$ external face, and $\mathcal D(c)$ is connected, then 
$\mathcal D(c')$ is 
connected.     
\end{claim}
\renewcommand{\qedsymbol}{$\blacksquare$}
\begin{proof}[Proof of Claim~\ref{claim:codim_1_connected}]
Since $\pi$ restricts to the identity on $c'$, and $\mathcal D(c')\subseteq\mathcal D(c)$, we have $\mathcal 
D(c')\subseteq 
\pi(\mathcal D(c))$.  We claim that $\pi(\mathcal D(c))\subseteq\mathcal D(c')$.  By definition, $\mathcal 
D(c)$ is the 
union of the completely external cubes of $c$, and $\mathcal D(c')$ is the union of the completely external 
cubes of $c'$.  
Hence it suffices to show that $\pi$ takes each completely external cube $f$ to a completely external cube.

If $f\cap c'\neq\emptyset$, then $f\cap c'$ is a completely external cube of 
$c'$.  On the other hand, $\pi(f)=f\cap c'$, so $\pi(f)$ is completely external.

Otherwise, $f$ is separated from $c'$ by some unique hyperplane $E$, and 
$\pi(f)$ is a cube of $c'$ crossing exactly those hyperplanes that cross $f$ and $c'$ and 
separated from $f$ by $E$.  If 
$\pi(f)$ is not completely external, then there is a $1$--cube $\pi(e)$ of 
$\pi(f)$ 
that is internal.  Now, since $f$ is completely external, every $1$--cube $e$ 
of $f$ parallel to $\pi(e)$ is external.  

For each codimension--$1$ face $c''$ containing $e$, there is no panel $P$ 
so that $P\cap c=c''$ and the abutting hyperplane for $P$ is dual to $e$.  By 
Lemma~\ref{lem:how_panels_intersect_cubes},\ref{lem:convex_panel_consequence}, 
there is a codimension--$1$ face 
$c'''$ and a panel $P$ so that $P\cap c=c'''$ and $\pi(e)$ itself is internal 
to $P$.  Since $E$ is the only hyperplane separating $f$ from $\pi(f)$, the 
face $c'''$ cannot cross $E$, for otherwise it would contain $e$.  Hence 
$c'''=c'$.  Hence every $1$--cube of $c'$ parallel to $e$ is internal, so $c'$ 
is internal, a contradiction.  

We have shown that $\pi$ restricts on $\mathcal 
D(c)$ to a surjection $\mathcal D(c)\to\mathcal D(c')$.  Since $\mathcal D(c)$ 
is connected and $\pi$ is continuous, it follows that $\mathcal D(c')$ 
is connected.
\end{proof}

\renewcommand{\qedsymbol}{$\Box$}

\textbf{Conclusion:}  Recall that $c'$ is an external subcube of $c$.  By Claim~\ref{claim:codim_1_connected}, 
if $d-d'=1$, 
then $\mathcal D(c')$ is connected.  So, assume that $d-d'\ge 2$.

Next, let $\mathcal E$ be the set of $1$--cubes $e$ of $c$ such that $c'\cap e\neq\emptyset$, and $c'$ does not 
contain $e$. 
 Note that for each $e\in\mathcal E$, there is a sub-cube $c'\times e$ of $c$.  Moreover, $c'\times e$ has 
codimension at 
least $1$ in $c$, because $c'$ has codimension at least $2$.

We claim that there exists an external $1$--cube $e\in \mathcal E$.  There are two cases to consider.  First, 
if $c'\cap 
h\neq\emptyset$, then by Lemma~\ref{lema:all_persistent}, $c'$ contains a persistent corner $v$ of $c$.  Any 
$1$--cube $e$ 
emanating from $v$ is external, by the definition of a persistent corner, and since such a $1$--cube can be 
chosen so as not 
to lie in $c'$, we see that $\mathcal E$ contains an external $1$--cube.  Second, suppose that $c'\cap 
h=\emptyset$ and 
suppose that every $1$--cube in $\mathcal E$ is internal.  Let $P$ be an edge-path in $c$ joining $h$ to $c'$.  
Then $P$ 
must pass through a $1$--cube in $\mathcal E$.  If every $1$--cube in $\mathcal E$ is internal, then $P$ cannot 
lie in 
$\mathcal D(c)$.  But since every $0$--cube of $c$ is in $\mathcal D(c)$, and $\mathcal D(c)$ is connected, 
this gives a 
contradiction.  Thus, in either case, $\mathcal E$ contains an external $1$--cube $e$.

Thus, by Lemma~\ref{claim:opposite}, $c'\times e$ is an external sub-cube of $c$ of codimension $d-d'-1$.  By 
induction on 
codimension, $\mathcal D(c'\times e)$ is connected.  Now, since $c'$ is a codimension--$1$ external subcube of 
$c'\times e$, 
Claim~\ref{claim:codim_1_connected} implies that $\mathcal D(c')$ is connected, as required.
\end{proof}

Now we can define $\mathcal F(c)$ for each cube $c$ of $\Psi$.  The idea is to construct a subspace of $c$ 
that contains all the completely external cubes (in particular, all the $0$--cubes) but does not contain the 
external cubes, and behaves well under intersections of cubes.  We will also want $\mathcal F(c)$ to have the 
structure of a locally CAT(0) cube complex and, in the case where $c$ is external, $\mathcal F(c)$ will 
actually be a deformation retract of $c$.  

The case where $c$ is internal plays only a small role, and is easy 
to deal with, so we handle it first:  

\begin{defn}[$\mathcal F(c)$ for internal $c$]\label{defn:internal}
Let $c$ be an internal cube.  Then $\mathcal F(c)=\mathcal D(c)$.  (Note that 
in this case, $\mathcal D(c)$ is disconnected by the abutting hyperplane for 
each panel in $\mathcal P$ in which $c$ is internal.)
\end{defn}

The definition of $\mathcal F(c)$ when $c$ is external is more complicated.  The reason is that $\mathcal 
D(c)$ can be disconnected, and in such situations, $\mathcal F(c)$ (which must be connected) cannot be taken 
to be a subcomplex.  Accordingly, we divide into cases according to whether $\mathcal D(c)$ is connected or 
not.  (Lemma~\ref{lem:connected_hereditary} above shows that this division into cases is stable under passing 
to subcubes, which will be important for later arguments by induction on dimension.)  The connected case is 
straightforward:

\begin{defn}[$\mathcal F(c)$ for external $c$, connected 
case]\label{defn:external_connected}
Let $c$ be an external cube for which $\mathcal D(c)$ is connected.  Then 
$\mathcal F(c)=\mathcal D(c)$.  This includes the case where $c$ is completely 
external, in which case $\mathcal F(c)=\mathcal D(c)=c$.
\end{defn}

Now let $c$ be an external cube that is not completely 
external.  

Lemma~\ref{lem:cube_intersect} provides a persistent corner $v\in 
c$.  Let $h(c)=h$ be the persistent subcube of $c$.  By Lemma~\ref{lema:all_persistent}, $h$ is completely external, and is 
thus $h\subsetneq c$.  Let $\bar h=\bar h(c)$ be the salient subcube of $c$.

If $h$ is a codimension--$1$ face, and $h=\bar h$, then every $1$--cube emanating from $h$ is external, so
every $0$--cube is connected to $h$ by an external $1$--cube, so $\mathcal D(c)$ is connected and $\mathcal F(c)=\mathcal 
D(c)$ by Definition~\ref{defn:external_connected}.

If $h$ is a codimension--$1$ face and $h\neq\bar h$, then there is a hyperplane $E$ separating $h,\bar h$ since $h,\bar h$ 
are parallel and distinct.  On the other hand, the definition of $\bar h$ implies that some $1$--cube dual to $E$ is 
internal.  But since $h$ is codimension--$1$, Lemma~\ref{lema:all_persistent} implies that every $1$--cube dual to $E$ is 
external, a contradiction.  

Hence we can assume that $\dimension c-\dimension h\ge2$.

For each codimension--$1$ face $c'$ of $c$ containing some persistent
corner $v_i$ of $c$, the cube $c'$ 
is spanned by a set of external $1$--cubes, so $c'$ is external.  By induction 
on dimension, we have defined $\mathcal F(c')$ to be a contractible 
subspace of $c'$ that contains $\mathcal D(c')$ and also contains the convex 
hull in $c'$ of all persistent corners of $c'$ (in particular, $\mathcal F(c')$ 
contains $h\cap c'$).  The base case is where $c$ is a 
$0$--cube, which is necessarily completely external, so $\mathcal F(c)=c$.

\textbf{$\mathcal F_0(c)$: Assembling pieces from codimension--$1$
  faces.} Now let $\faces'(c)$ be the set of codimension--$1$
sub-cubes $c'$ of $c$ containing a persistent corner $v_i$.  Let
$\mathcal F_0(c)=\bigcup_{c'\in\faces'(c)}\mathcal F(c')$.  Since $h$
is connected and intersects $\mathcal F(c')$ for each
$c'\in\faces'(c)$, the subspace $\mathcal F_0(c)$ is connected.  (Later, in Lemma~\ref{lem:compat}, we will see that 
$\faces'(c)$ is precisely the set of external codimension--$1$ faces of $c$.)

\textbf{$\mathcal F_1(c)$: Connecting subcubes of $h$ to their opposites in $\bar
  h$.} Recall that $\bar h$ is the parallel copy of $h$ in $c$ that is separated
from $h$ by exactly those hyperplanes $H$ so that:
\begin{itemize}
     \item $H$ does not cross $h$;
     \item at least one $1$--cube dual to $H$ is internal.
\end{itemize}
Let $\kappa\leq d-\dimension h$ be the number of hyperplanes
separating $h,\bar h$.  For each completely external cube
$w\subseteq\bar h$, let $\mathcal S(w)$ be the $\ell_2$--convex hull of $w$ and
the cube $\bar w$ (a sub-cube of $h$) parallel to $w$ and separated
from $w$ by the above hyperplanes.

Equivalently, $c$ has an $\ell_2$--convex subspace
$h\times\left[-\frac{\sqrt{\kappa}}{2},\frac{\sqrt{\kappa}}{2}\right]$,
intersecting exactly those hyperplanes that separate the persistent
and salient cubes, and for each completely external subcube $w$ of
$\bar h = h\times\{\frac{\sqrt{\kappa}}{2}\}$, let
$\mathcal S(w)=w\times\left[-\frac{\sqrt{\kappa}}{2},\frac{\sqrt{\kappa}}{2}\right]$. Also
note that $\dimension \mathcal S(w)=\dimension w + 1\leq d-1$ since
$\dimension h\leq d-2$ and $w\subseteq h$.

Form $\mathcal F_1(c)$ from
$\mathcal F_0(c)$ by adding $\mathcal S(w)$ for each
completely external cube $w$ of $\bar h$.

\begin{rem}\label{rem:nonempty_S}
Observe that, if $\mathcal D(c)$ is disconnected, then there exists $w$ as above with $\mathcal S(w)\neq\emptyset$.  Indeed, 
the persistent subcube $h$ is nonempty and $\bar h$ is therefore nonempty, since it is parallel to $h$.  Any $0$--cube $w$ 
of $\bar h$ is completely external, so by definition $\mathcal S(w)$ is a nonempty subspace containing $w$.  Moreover, 
$\mathcal S(w)$ is contained in $\mathcal F_1(c)$.  
\end{rem}

\textbf{$\mathcal{F}(c)$: Adding the missing cubes from
  $\mathcal{D}(c)$.}  To complete the definition of $\mathcal F(c)$,
we have to ensure that it contains $\mathcal D(c)$, in order to
support the above inductive part of the definition.  In other words, we need to
add to $\mathcal F_1(c)$ any completely external sub-cube of $c$ that
does not already appear in $\mathcal F_1(c)$.  Hence let $\mathcal F(c)$ be the union of $\mathcal F_1(c)$ and any
totally external cube of $c$ not already contained in
$\mathcal F_1(c)$.  Some examples are shown in Figure~\ref{fig:fundamentz}.

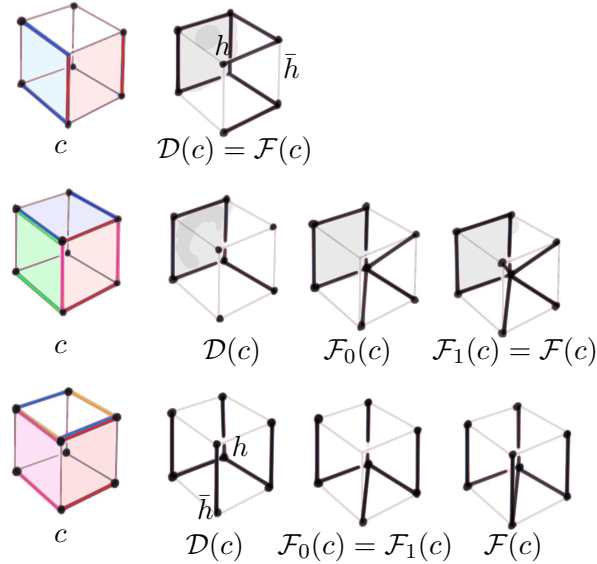
\begin{figure}[h]
\begin{overpic}[width=0.5\textwidth]{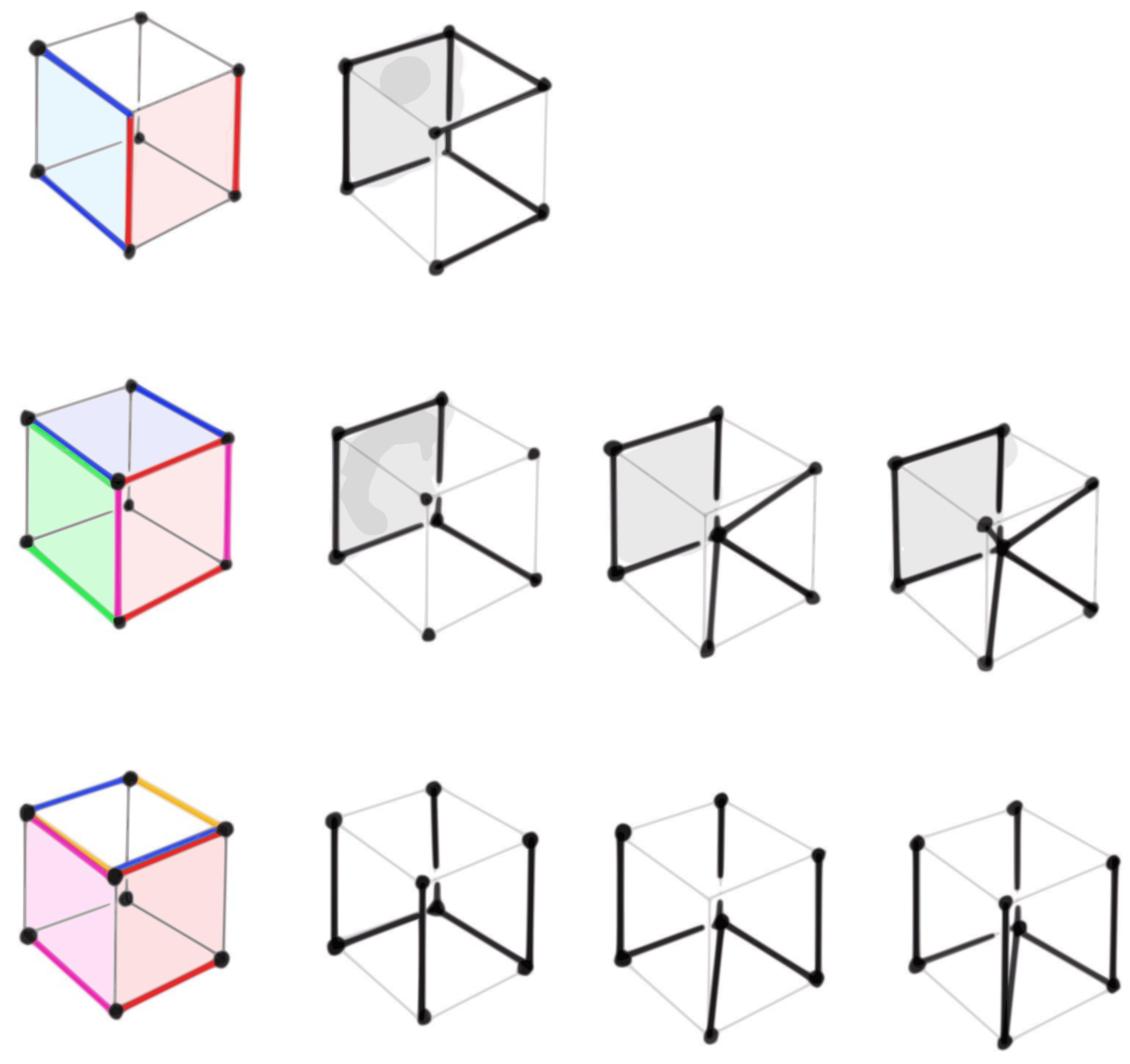}
 \put(9,-1){$c$}
 \put(9,32){$c$}
 \put(9,67){$c$}
 \put(27,66){$\mathcal D(c)=\mathcal F(c)$}
 \put(35,31){$\mathcal D(c)$}
 \put(56,31){$\mathcal F_0(c)$}
 \put(75,31){$\mathcal F_1(c)=\mathcal F(c)$}
 \put(32,-3){$\mathcal D(c)$}
 \put(48,-3){$\mathcal F_0(c)=\mathcal F_1(c)$}
 \put(84,-3){$\mathcal F(c)$}
 \put(37,84){$h$}
 \put(49,80){$\bar h$}
 \put(40,14){$h$}
 \put(34,3){$\bar h$}
\end{overpic}
\caption{Three examples of the construction of $\mathcal F(c)$ from a set of panels on $c$, in the case $\dimension c=3$.  In 
each case, the faces of $c$ belonging to panels in $\mathcal P$ are shaded, and the internal $1$--cubes of $c$ are bold.  The 
remaining pictures show how $\mathcal D(c)$, $\mathcal F_0(c)$, $\mathcal F_1(c)$, and $\mathcal F(c)$ sit inside $c$.  The 
$1$--cubes of $\mathcal F(c)$ (which are either $1$--cubes of $c$ or diagonals in $c$ or its faces) are bold.  In the first 
and last cases, $h,\bar h$ are labelled; in the middle case, $h$ is the unique $0$--cube with $3$ incident external 
$1$--cubes, and $\bar h$ is diagonally opposite $h$ in $c$.}\label{fig:fundamentz}
\end{figure}

\subsection{Basic properties of the fundament}\label{subsubsec:basic_properties}
Before proceeding to a more technical analysis of $\mathcal F(c)$, we collect some important properties that 
capture most of the essence of why $\mathcal F(c)$ has been defined as it has.  

First, by definition, $\mathcal F(c)$ contains every $0$--cube of $c$, since $0$--cubes are completely 
external and $\mathcal D(c)\subseteq\mathcal F(c)$.

Second, $\mathcal F(c)$ is ``completely external'' in the following sense:

\begin{lem}\label{lem:no_interior}
For each cube $c$ and each $P\in\mathcal P$, we have $\mathcal 
F(c)\cap\interior{P}=\emptyset$.
\end{lem}

\begin{proof}
By definition, $\mathcal D(c)$ is disjoint from $\interior{P}$.  On the other 
hand, any cube of $\Psi$ containing some $\mathcal S(w)$ is external.
Thus the 
lemma follows from the definition of $\mathcal F(c)$.
\end{proof}

Third, each $\mathcal F(c)$ was defined in terms of $\mathcal P$ (and the set 
of $1$--cubes internal to panels in $\mathcal P$), from which we immediately get:

\begin{lem}\label{lem:canonical}
For each cube $c$, the subspace $\mathcal F(c)$ is uniquely determined by the set of internal $1$--cubes of 
$c$.  In particular, if $g\in\Aut(\Psi)$ preserves $\mathcal P$ (taking insides to insides), then $\mathcal 
F(gc)=g\mathcal F(c)$ for all cubes $c$.
\end{lem}

Finally, the fundament has been defined so as to be consistent under passing to subcubes:

\begin{lem}[Fundaments of subcubes]\label{lem:compat}
Let $c$ be an external cube.  Then both of the following hold:  
\begin{enumerate}
 \item For each external sub-cube 
$c'$ of $c$, $$\mathcal F(c)\cap c'=\mathcal F(c').$$\label{item:general}
\item If $c'$ is a codimension--$1$ external face of 
$c$, then $c'\in\faces'(c)$.\label{item:codim_1} 
\end{enumerate}
\end{lem}

(Recall that $\faces'(c)$ denotes the set of codimension--$1$ faces of $c$ that contain persistent corners of 
$c$.)

 Lemma~\ref{lem:compat} has the following important consequence.  Let $b,c$ be 
external cubes.  If $b\cap c$ is external, then by the lemma, $\mathcal 
F(c)\cap b=\mathcal F(c)\cap(b\cap c)=\mathcal F(b\cap c)$, and the same is 
true with the roles of $b$ and $c$ reversed.  Hence $\mathcal F(c)\cap\mathcal 
F(b)=\mathcal F(b\cap c)$.

\begin{proof}[Proof of Lemma~\ref{lem:compat}]
We will prove the second assertion first, and use it in the proof of the first assertion.

\textbf{Proof of assertion~\eqref{item:codim_1}:}  Suppose that 
$c'\not\in\faces'(c)$ and $c'$ has codimension $1$.   Since 
$\faces'(c)$ contains a maximal collection of pairwise-intersecting 
codimension--$1$ sub-cubes of $c$ (namely, all the faces containing any given 
persistent corner $v_i$), there exists $\bar c'\in\faces'(c)$ so that $\bar c'$ 
and $c'$ are 
parallel.  Now, if some persistent corner $v_i\in c'$, then $c'\in\faces'(c)$, 
a contradiction.  
So each persistent corner $v_i\in\bar c'$, whence $h=h(c)\subset\bar c'$.

Let $v'\in c'$ be a $0$--cube connected by a $1$--cube to some $v_i$ (this 
exists since $c'$ is codimension--$1$).  The $1$--cube joining $v_i,v'$ is 
external, because $v_i$ is a persistent corner.  Hence, since 
$c'\not\in\faces'(c)$, there exists a $1$--cube $e$ of $c'$ that is incident to 
$v'$ and internal to some $P\in\mathcal P$.

Lemma~\ref{lem:how_panels_intersect_cubes} provides a codimension--$1$ face 
$c''$ of $c$ so that $e\subset c''$ and every $1$--cube of $c''$ parallel to 
$e$ 
is internal to $P$.  Now, $c''$ cannot contain the $1$--cube $\bar e$ parallel 
to $e$ and incident to $v_i$, since $v_i$ is a persistent corner.  Hence $c''$ 
is 
disjoint from $\bar c'$.  (Because any codimension--$1$ face containing $e$ and 
intersecting $\bar c'$ must contain $\bar e$.)  Thus $c''=c'$, so $c'$ is 
internal, a contradiction.  Thus $c'\in\faces'(c)$.

\textbf{Proof of assertion~\eqref{item:general}:}  We first prove the claim in the case
where $c'$ is codimension--$1$. If $c$ is completely external, then so
is $c'$, so $\mathcal F(c)\cap c'=c\cap c'=c'=\mathcal F(c')$.

Now suppose that $c$ is not completely external.  If $\mathcal D(c)$ is 
connected, then by definition, $\mathcal F(c)=\mathcal D(c)$.  
Lemma~\ref{lem:connected_hereditary}  implies 
that $\mathcal D(c')$ is connected, so $\mathcal F(c')=\mathcal D(c')$.  It 
follows immediately from the definition that $\mathcal D(c')=\mathcal D(c)\cap 
c'$, so $\mathcal F(c)\cap c'=\mathcal F(c')$.

It remains to consider the case where $\mathcal D(c)$ is disconnected.  By the 
first part of the lemma, $c'\in\faces'(c)$.  By definition, $\mathcal 
F(c)\cap 
c'\supseteq\mathcal F(c')$.  Lemma~\ref{lem:overlap} below implies that for 
any 
completely external cube $w$ in $\bar h=\bar h(c)$ so that $\mathcal S(w)$ intersects 
$c'$, we have $\mathcal S(w)\cap c'\subseteq \mathcal S(w\cap c')\subseteq\mathcal F(c')$.  Now, 
either $\mathcal F(c)\cap c'=\mathcal D(c)\cap c'=\mathcal D(c')$ or $\mathcal 
F(c)\cap c'=(\mathcal D(c)\cap c')\cup(\bigcup_w\mathcal S(w)\cap c')$, so in either 
case $\mathcal F(c)\cap c'\subseteq\mathcal F(c')$, as required.

Now we complete the proof by arguing by induction on codimension $k=\dimension c-\dimension 
c'$.  The case $k=1$ was done above.  Let $k\geq1$.  If $c'$ is not contained 
in some element of $\faces'(c)$, then $c'$ is a completely external
cube that gets added in passing from $\mathcal F_1(c)$ to $\mathcal
F(c)$.

Indeed, $\mathcal F(c)$ is the union of $\mathcal F_0(c)$ together
with some completely external cubes and some subspaces of the form
$\interior{\mathcal S(w)}$, and the latter are disjoint from $c'$
unless $c'$ is contained in an element of $\faces'(c)$.  So, if $c'$
is external and does not lie in an element of $\faces'(c)$, then
$\mathcal F(c')=c'\supset\mathcal F(c)\cap c'$.

Otherwise, $c'$ is a subcube of some $l\in\faces'(c)$.  Note 
that $\dimension l-\dimension c'<k$ and $l$ is external, so by induction on 
codimension, $\mathcal F(l)\cap c'=\mathcal F(c')$.  But by the case $k=1$, we 
have $\mathcal F(c)\cap l=\mathcal F(l)$.  So $\mathcal F(c)\cap c'=\mathcal 
F(c)\cap (l\cap c')=\mathcal F(l)\cap c'=\mathcal F(c')$, as required.
\end{proof}

The next lemma supported the previous one:

\begin{lem}\label{lem:overlap}
Let $c$ be an external cube and let $h=h(c),\bar h=\bar h(c)$.  Let 
$w$ be a completely external cube of $\bar h$, and let $c'\in\faces'(c)$.  Then 
$\mathcal S(w)\cap c'\subseteq\mathcal F(c')$. 
\end{lem}

\begin{proof}
Since $c'$ is external, Lemma~\ref{lem:cube_intersect} provides a nonempty 
persistent subcube $h'=h(c')$ of $c'$.  As 
before, let $\bar h'=\bar h(c')$ be the salient subcube of $c'$.  
(Recall that $\bar h'$ is separated from $h'$ by those hyperplanes that do 
not cross $h'$ and which have at least one dual internal $1$--cube in $c'$.)

Note that $h(c)\cap c'\subset h'$ and
$\bar h(c)\cap c'\subset\bar h'$. Note furthermore that $h \cap c'$ is
the projection of $h$ to $c'$. Also, each hyperplane $D'$ in $c'$
separating $h',\bar h'$ extends to a hyperplane $D$ of $c$ separating
$h(c),\bar h(c)$. Suppose this was not the case, since $D'$ is dual to
an internal 1-cube, $D$ must as well. It follows that if $D$ fails to
separate $h,\bar h$ then it must $D$ must intersect $h$, but then its
projection $D'$ must intersect the $h\cap c' \subset h'$, which is a
contradiction.

Now suppose that  $\bar h(c)\cap c'=\emptyset$.  Let $E$ be a hyperplane separating $\bar h(c)$ from $c'$ (so, 
$E$ is the 
unique hyperplane parallel to the codimension--$1$ face $c'$ of $c$).  Since $E$ does not cross $\bar h(c)$, 
and $\bar h(c)$ 
is parallel to $h(c)$, we see that $E$ also does not cross $h(c)$.  Hence, if $E$ is dual to an internal 
$1$--cube, then by 
definition $E$ separates $\bar h(c)$ from $h(c)$, and thus $h(c)\subseteq c'$.  In this case, $\mathcal 
S(w)\cap c'=\bar w$, 
the completely external subcube of $h(c)$ diagonally opposite $w$.  On the other hand, $\bar w\subset h(c)\cap 
c'\subseteq h'\subset\mathcal F(c')$, so the lemma holds in this case.

If $E$ is not dual to an internal $1$--cube, then $\bar h(c)$ and $h(c)$ are not separated by $E$, so 
$h(c)\cap 
c'=\emptyset$.  But $h(c)$ contains all persistent corners of $c$, so $c'$ contains no persistent corners of 
$c$.  
Since $c\in\faces'(c)$, this is a contradiction. 

Otherwise, if $\bar h(c)\cap c'\neq\emptyset$, then $h(c)\cap c'$ 
and 
$\bar h(c)\cap c'$ are diagonally opposite in $c'$ in the above sense, and 
$\mathcal S(w)\cap c'=\mathcal S(w\cap 
c')\subset\mathcal F(c')$.
\end{proof}

\subsection{Relationship between $\mathcal F_1(c)$ and $\mathcal F(c)$}  Our goal is to create a CAT(0) cube 
complex $\Psi_\bullet$ from $\Psi$ by deformation retracting each external cube $c$ to its fundament, by 
induction on dimension.  In the case where $\mathcal D(c)$ is disconnected, this requires an understanding of 
the relationship between $\mathcal F(c)$ and $\mathcal F_1(c)$.  Lemma~\ref{lem:extra_cubes} explains this 
relationship.  In
particular, it gives a precise description
of how $\mathcal F(c)$ is constructed: it states that all new cubes
added to $\mathcal F_1(c)$ are completely external and contain the salient subcube $\bar h$, which must itself be completely 
external in this case (in general, it can happen that $\mathcal F(c)=\mathcal F_1(c)$ and $\bar h$ is not completely 
external).

The main import of the lemma is the final statement about deformation retractions, which will be used in 
Lemma~\ref{lem:external_crush} below (which is where we construct the deformation retraction from $c$ to 
$\mathcal F(c)$).

\begin{lem}[$\mathcal F_1(c)$ versus $\mathcal F(c)$]\label{lem:extra_cubes}
Suppose that $c$ is external.  Then  $\mathcal D(c)\subset\mathcal F(c)$.  If, in addition, $\mathcal D(c)\ne c$, then all 
of the following hold, where $h=h(c)$ and 
let $\bar h=\bar h(c)$:
\begin{enumerate}
     \item \label{item:codim} If $f$ is a completely external cube 
of 
$\mathcal F(c)$ that is not contained in $\mathcal F_0(c)$, then $f$ has 
codimension at least $2$ in $c$.
\item \label{item:bar_h}If $f$ is a maximal completely external subcube of $c$ 
lying in $\mathcal F(c)$, then either $f\subset\mathcal F_1(c)$ or $f$ contains 
$\bar h$.  In the latter case, $\bar h$ must be completely external.

\item\label{item:l_or_equal}  If $\mathcal D(c)$ is disconnected, then either 
$\mathcal F_1(c)=\mathcal F(c)$, or $\bar h$ 
is completely external and $h\cup \bar h$ is contained in a codimension--$1$ face of $c$.  If $\mathcal D(c)$ 
is connected, 
then $h\cup \bar h$ is contained in a codimension--$1$ face of $c$.

\item \label{item:proper_containment_restriction}Let $f_1,\ldots,f_k$ be the 
maximal completely external cubes of $c$ that do not lie in $\mathcal F_1(c)$.  
Then:
\begin{itemize}
 \item $k\le 1$.
 \item $f_1=\bar h\times s$, where $s$ is a subcube of $c$, none of whose 
hyperplanes is dual to an internal $1$--cube of $c$.
 \item There exists a cube $\bar h'\subset f_1$, parallel to $\bar h$, such that $\bar h'$ is not contained in 
any $c'\in\faces'(c)$.  Hence 
$\interior{\bar h'}$ is disjoint from $\bigcup_{c'\in\faces'(c)}c'$.  (Here, 
$\interior{\bar h'}$ is understood to denote $\bar h'$ if $\dimension\bar h'=0$ 
and otherwise denotes the open cube of $c$ whose closure is $\bar h'$.)
\end{itemize}

\end{enumerate}
In particular, if  $\mathcal D(c)$ is disconnected (i.e. $\mathcal 
F(c)\neq\mathcal D(c)$), then there is a strong deformation retraction from 
$\left(\bigcup_{c'\in\faces'(c)}c'\right)\cup\mathcal S\cup\bigcup_if_i$ to 
$\left(\bigcup_{c'\in\faces'(c)}c'\right)\cup\mathcal S$, fixing each $c'$, 
where $\mathcal S$ is the union of the $\mathcal S(w)$ over the completely 
external cubes $w$ of $\bar h$.
\end{lem}

\begin{proof}
By definition, $\mathcal D(c)$ is the union of all completely external cubes of 
$c$.  Each such $f\subset c$ either lies in some element of $\faces'(c)$, and 
hence in $\mathcal F_1(c)$ (by the inductive definition), or is added to 
$\mathcal F_1(c)$ when constructing $\mathcal F(c)$.  This proves $\mathcal D(c)\subseteq\mathcal F(c)$ (by construction, 
the containment is an equality if and only if $\mathcal D(c)$ is connected).  

Now suppose $\mathcal D(c)\neq c$.

\textbf{Assertion~\eqref{item:codim}:}  Suppose that the completely external subcube $f$ of $c$ is not contained in any 
$c'\in\faces'(c)$. If $f$ is codimension--$1$, this amounts to saying that $f$ 
has no persistent corner of $c$.  Let $E$ be the hyperplane of $c$ not crossing $f$.  Since $f$ is 
completely external, each of its $0$--cubes is a persistent corner \emph{of $f$}, so each $1$--cube dual to 
$E$ is internal (else $f$ would contain a persistent corner of $c$).  Hence $c$ is internal, a contradiction.  

\textbf{Assertion~\eqref{item:bar_h}:}  Assume that $f$ is a maximal
completely external subcube of $c$ that is not contained in $\mathcal F_1(c)$.

We need some preparatory claims:  

\begin{claim}\label{claim:separate_f_h}
The cube $h$ is parallel to a sub-cube of $f$, and any hyperplane $H$ with $H\cap f=\emptyset$ must separate $f$ from $h$.  
\end{claim}
\renewcommand{\qedsymbol}{$\blacksquare$}
\begin{proof}[Proof of Claim~\ref{claim:separate_f_h}]
Let $H$ be a hyperplane crossing $h$, so that by the definition of $h$, there are persistent corners $v,v'$ separated by 
$H$.  Let $k,k'$ be codimension--$1$ faces of $c$, containing $v,v'$
respectively, with $k,k'$ parallel to $H$. On the one hand $k,k'$
contain persistent corners, and therefore lie in $\faces'(c)$. On the
other hand if $H \cap f = \emptyset$, then $f$ must be a subcube of
either $k$ or $k'$ contradicting $f \not\subset \mathcal F_1(c)$.  Thus $H\cap f\neq\emptyset$.  This shows that, for all hyperplanes $H$, if $H\cap f=\emptyset$, then $H\cap 
h=\emptyset$.  Since every hyperplane crossing $h$ crosses $f$, it follows that $h$ is parallel to a sub-cube $h'$ of $f$.

Moreover, if $H\cap f=\emptyset$, and $h,f$ are not separated by $H$, then $f$ lies in a codimension--$1$ face of $c$ 
containing a persistent corner.  Hence $f\subset\mathcal F_0(c)\subseteq\mathcal F_1(c)$, a contradiction.  Thus, if $H\cap 
f=\emptyset$, then $H$ separates $f$ from $h$.
\end{proof}

\begin{claim}\label{claim:disjoint_internal}
Let $H$ be a hyperplane of $c$ such that $H\cap f=\emptyset$ and $H\cap h=\emptyset$.  Then $H$ is dual to an internal 
$1$--cube of $c$.  Hence $H$ separates $h$ from $\bar h$.
\end{claim}

\begin{proof}[Proof of Claim~\ref{claim:disjoint_internal}]
The second conclusion follows from the first, by the definition of $\bar h$, so it suffices to prove the first.  Suppose 
that every $1$--cube dual to $H$ is external.  By Claim~\ref{claim:separate_f_h}, $H$ separates $f$ and $h$.

Fix a $1$--cube $e$ dual to $H$ and incident to a $0$--cube of $f$.
Since $e$ is not in $f$, there is a subcube $e\times f$ of $c$,
properly containing $f$.  (Since $H\cap f=\emptyset$, the dual
$1$--cube $e$ cannot be a $1$--cube of $f$.) Let $k,k'$ be the
codimension--$1$ faces of $c$ parallel to $H$, with $h\subseteq k$ and
$f\subseteq k'$.  Let $e'$ be a $1$--cube of $e\times f$.

If $e' \subset f\times e$ is parallel to $e$, then we can assume that
$e'$ is external since it is dual to $H$. Otherwise, if
$e'\subset (f\times e)\cap f = f \cap k'$, then $e'$ is external since
$f$ is completely external.

The remaining case is where $e'\subset k\cap (f\times e)$, so $e'$ is
parallel to a $1$--cube $\bar e'$ of $f$.  We will choose a specific
such $\bar e'$ momentarily.

First, let $e''$ be a $1$--cube parallel to $e'$ that contains a
$0$--cube of $h$.  By Lemma~\ref{lema:all_persistent}, $e''$ contains
a persistent corner and is therefore external.

Now choose $\bar e'$
as above (so, $\bar e'$ is parallel to $e'$ and $e''$ and lies in $f$)
such that the distance from $e''$ to $\bar e'$ is maximal over all
possible such choices.

We saw already that $e''$ is external.  Since $f$ is completely external and $\bar e'$ is contained in $f$, we also have 
that $\bar e'$ is external.

We claim that $e''$ and $\bar e'$ are not contained in a common
codimension--$1$ face of $c$.  Suppose to the contrary that they are.
Let $D$ be the hyperplane parallel to the codimension--$1$ face
containing $e'',\bar e'$.  Now, if $D\cap f=\emptyset$, then
$D\cap h=\emptyset$.  Hence, since $e''$ contains a $0$--cube of $h$,
and $\bar e'$ lies in $f$, and $e''$ and $\bar e'$ lie on the same
side of $D$, we see that $D$ does not separate $h$ from $f$,
contradicting Claim~\ref{claim:separate_f_h}.  Thus $D$ crosses $f$.
But then $f$ contains a parallel copy of $\bar e'$ that is separated
from $\bar e'$ by $D$, and no other hyperplane; this parallel copy is
thus further from $e''$ than $\bar e'$ is, contradicting how $\bar e'$
was chosen.  Thus $e''$ and $\bar e'$ do not lie in a common
codimension--$1$ face.

Hence, by 
Lemma~\ref{lem:how_panels_intersect_cubes}, every $1$--cube parallel to $\bar e'$, and in particular $e'$, is external (see 
Remark~\ref{rem:codim_1_external}).  Thus $e\times f$ is a completely external cube, contradicting 
maximality of $f$.  We conclude that every hyperplane $H$ separating $f,h$ is dual to an internal $1$--cube.
\end{proof}

We can now conclude the proof of assertion~\eqref{item:bar_h}.  Recall that we have assumed that $\bar h$ is 
not contained in $\mathcal F_1(c)$.  Note that $\bar h$ is contained in $f$ if and only if, for each 
hyperplane $H$, if $H$ and $f$ are disjoint, then $\bar h$ and $f$ are on the same side of $H$.  Indeed, 
obviously $\bar h\subseteq f$ implies that no hyperplane separates $\bar h,f$.  Conversely, suppose that no 
hyperplane separates $\bar h$ from $f$.  Then $\bar h\cap f\neq\emptyset$, so $\bar h\cap f=\bar h$ since 
$\bar h$ is parallel to $h$ and hence parallel to a subcube of $f$, by Claim~\ref{claim:separate_f_h}.

Now, by 
Claim~\ref{claim:separate_f_h}, if $H$ is a hyperplane disjoint from $f$, then $H$ must separate $f$ from 
$h$.  By Claim~\ref{claim:disjoint_internal}, $H$ separates $h$ from $\bar h$.  Hence $f,\bar h$ are on the 
same side of $H$ (namely, the halfspace not containing $h$).  Thus $\bar h\subseteq f$.  Since $f$ is 
completely external by hypothesis, the same is true of $\bar h$.  This proves assertion~\eqref{item:bar_h}.

\textbf{Assertion~\eqref{item:proper_containment_restriction}:} Now let 
$f_1,\ldots,f_k$ be the maximal completely external cubes of $c$ that do not lie in $\mathcal F_1(c)$.  If $k=0$, there is 
nothing to prove, so suppose $k\geq 1$.  By assertion~\eqref{item:bar_h}, $\bar h$ is completely external and $\bar 
h\subseteq f_i$ for all $i$.

\begin{claim}\label{claim:proper_face}
If $\mathcal D(c)$ is connected and $\mathcal D(c)\neq c$, then $c$ has a codimension--$1$ face $l$ containing $h\cup\bar h$.
\end{claim}

\begin{proof}[Proof of Claim~\ref{claim:proper_face}]
Let $H$ be a hyperplane.  Suppose that $H\cap h=\emptyset$, so $H\cap \bar h=\emptyset$.  If $H$ is not dual to an internal 
$1$--cube, then $H$ does not separate $h$ from $\bar h$, and thus $h\cup\bar h$ lie in a common codimension--$1$ face.  
Hence we can assume that for each $H$, either $H$ intersects $h$ (and thus $\bar h$) or $H$ is dual to some internal 
$1$--cube.

We can assume there is at least one hyperplane $H$ disjoint from $h$ (and thus dual to an internal $1$--cube).  Otherwise, 
by the above, every hyperplane intersects $h$, so $\mathcal D(c)=c$ and $h=\bar h$.

We now show that the above two assumptions imply that $\mathcal D(c)$ is disconnected, yielding a contradiction.  

Let $v\in\bar h$ be a $0$--cube.  Then there is a unique $0$--cube $v'$ such that every hyperplane of $c$ separates $v$ from 
$v'$.  We claim that $v'$ is a persistent corner.

Let $w$ be the closest $0$--cube (in the graph metric on $c^{(1)}$) of $h$ to $v$.  The $0$--cube $w$ is characterised by 
the property that a hyperplane $H$ separates $v$ from $w$ if and only if $H$ separates $v$ from $h$.  So, the hyperplanes 
separating $v$ from $v'$ fall into two categories: those that separate $v$ from $h$, and those that separate $w$ from $v'$ 
but do not separate $v$ from $h$.  Since $w\in h$, and each hyperplane separating $w$ from $v'$ crosses $h$, we thus have 
$v'\in h$.  By Lemma~\ref{lema:all_persistent}, $v'$ is a persistent corner, as required.

Let $e$ be a $1$--cube that contains $v$ and does not lie in $\bar h$.  Let $e'$ be the parallel $1$--cube containing $v'$.  
Since every hyperplane not dual to $e$ separates $e,e'$, and $e'$ is external (since $v'$ is a persistent 
corner), if $e$ is external then so is every $1$--cube parallel to $e$, by Lemma~\ref{lem:how_panels_intersect_cubes} (see 
Remark~\ref{rem:codim_1_external}).  Hence, since $e$ is dual to a hyperplane separating $h,\bar h$, and this hyperplane must 
be dual to an internal $1$--cube, $e$ is internal.  

Now, any edge-path $\sigma$ from $h$ to $\bar h$ must pass through a $1$--cube that contains some $v\in\bar h$ and does not 
lie in $\bar h$.  We saw above that such a $1$--cube must be internal.  Hence $\sigma$ is not in $\mathcal D(c)$.  So, since 
$h,\bar h\subset\mathcal D(c)$, $\mathcal D(c)$ is disconnected, as required.
\end{proof}

In view of the preceding claim,  assume that $\mathcal D(c)$ is disconnected. Then $\mathcal F_1(c)$ contains 
$\mathcal S(\bar h)$ since $\bar h$ is completely external.  So, $f_1$ must properly contain $\bar h$, for 
otherwise we would have $f_1\subset\mathcal F_1(c)$.  Hence, since each $f_i$ is a maximal completely external cube, 
$f_i\not\subset f_1$ and hence $\bar h\subsetneq f_i$.  

Now, if some hyperplane $E$ satisfies $E\cap 
f_i=\emptyset$ and does not separate $h,f_i$, then $h,f_i$, and hence $h,\bar h$, are on the same side of $E$, and we are 
done.  Hence we can assume that every hyperplane either crosses $f_i$ or separates $h,f_i$.

\begin{claim}\label{claim:everyone_external}
Let $D_i$ be a hyperplane crossing $f_i$ but disjoint from $\bar h$.  Then every $1$--cube dual to $D_i$ is 
external.
\end{claim}

\begin{proof}[Proof of Claim~\ref{claim:everyone_external}]
Let $v$ be a persistent corner, and let $v'$ be the $0$--cube separated from $v$ by all 
hyperplanes.  Since every hyperplane of $c$ either separates $h,f_i$ or crosses $f_i$, we have $v'\in f_i$.  Now, if $e'$ is 
a $1$--cube containing $v'$ and 
$e$ is parallel to $e'$ and containing $v$, then $e,e'$ are not contained in a common codimension--$1$ face of $c$.  Choose 
$e,e'$ to be dual to $D_i$.  Now, $e$ is external since $v$ is a persistent corner.  On the other hand, $e'\subset f_i$, so 
$e'$ is external.  Thus $e,e'$ are external $1$--cubes dual to $D_i$ that do not lie in a common codimension--$1$ face of 
$c$.  Hence, by Lemma~\ref{lem:how_panels_intersect_cubes} (via Remark~\ref{rem:codim_1_external}), every $1$--cube dual to 
$D_i$ is external.
\end{proof}
\renewcommand{\qedsymbol}{$\Box$}

From Claim~\ref{claim:everyone_external}, it follows that $h$ and
$\bar h$ lie on the same side of $D_i$, whence there is a
codimension--$1$ face containing $h\cup\bar h$.  Thus, whether or not
$\mathcal D(c)$ is connected, there is a codimension--$1$ face $l$
containing $h\cup\bar h$ provided $k\ge1$. Let $l'$ be the smallest
subcube of $c$ containing $h\cup\bar h$, i.e.  $l'$ is the
intersection of all codimension--$1$ faces $l$ as above.

Hence $c=l'\times s$, where $s$ is a cube with the property that, for
every hyperplane $D$ crossing $s$, all $1$--cubes dual to $D$ are
external.  Indeed, any hyperplane dual to an internal $1$--cube either
separates $h,\bar h$ or crosses $h$, and thus any such hyperplane
crosses $l'$.

Fix $i$.  Since $f_i\cap l'$ is a subcube of $f_i$, we can write $f_i=(f_i\cap l')\times s'$, where $s'$ is some cube 
intersecting $f_i\cap l'$ in a single $0$--cube.  Since the image of $f_i$ 
under the canonical projection $c\to l'$ is $f_i\cap l'$, $s'$ has 
trivial projection to $l'$, i.e. $s'$ is a subcube of $s$.

Now, suppose that $s'$ is a proper subcube of $s$, so that $s=s'\times s''$ for some nontrivial cube $s''$.  Let $E$ be a 
hyperplane crossing $s''$.  Then $E$ does not separate $h,\bar h$, since $E$ does not cross $l'$.  Since $E$ does not cross 
$f_i$, we thus have that $f_i$ lies on the side of $E$ containing $\bar h$, and hence there is a codimension--$1$ face, 
parallel to $E$, that contains $f_i$ and $h$.  Hence $f_i$ is contained in an element of $\faces'(c)$, a contradiction.  We 
conclude that $f_i=(f_i\cap l')\times s$.

For each $i$, we have
$f_i\cap l'=\bar h$.  Indeed, $\bar h\subseteq f_i\cap l'$ by our
earlier discussion.  On the other hand, any hyperplane crossing $l'$
but not $\bar h$ must be dual to an internal $1$--cube, while by
Claim~\ref{claim:everyone_external}, every hyperplane crossing $f_i$
but not $\bar h$ fails to be dual to an internal $1$--cube.  Hence
every hyperplane crossing $f_i\cap l'$ crosses $\bar h$, so
$f_i\cap l'=\bar h$.  Thus, for all $i$,
$f_i=(f_i\cap l')\times s=\bar h\times s$.  So, $k=1$, and we have a
single ``extra'' completely external cube $f_1$ that is maximal with
the property that it is not contained in any element of $\faces'(c)$.

Let $\bar h'$ be the parallel copy of $\bar h$ in $f_1$ that is
separated from $\bar h$ by precisely those hyperplanes crossing
$s$. Let $c'$ be a codimension--$1$ face of $c$ containing $\bar h'$.
Let $E$ be the hyperplane parallel to $c'$. If $E$ is dual to a
$1$--cube of $l'$, and $E$ separates $h$ from $\bar h$, then $E$
separates $h$ from $\bar h'$ and, in particular $E$ separates $c'$
from $h$, so $c'\not\in\faces'(c)$.

If $E$ is dual to a $1$--cube of $s$, then $E$ cannot separate
$h,\bar h$, since all $1$--cubes dual to $E$ are external.  On the
other hand, $E$ separates $\bar h'$ from $\bar h$, and hence from
$h$. Thus $E$, as above, separates $h$ from $c'$, so $c'$ is not in
$\faces'(c)$.  Otherwise, if $E$ crosses $h$, then $E$ crosses
$\bar h$ and hence crosses $\bar h'$.  Thus $\bar h'\not\subset c'$.

Thus $\bar h'\subset\bigcap_if_i=f_1$ and $\bar 
h'\not\subset c'$ for any $c'\in\faces'(c)$.  Let $\interior{\bar h'}$ be the 
open cube whose closure is $\bar h'$ (or, if $\bar h'$ is a $0$--cube, let 
$\interior{\bar h'}=\bar h'$).  If $\interior{\bar h'}$ is contained in 
$\bigcup_{c'\in\faces'(c)}c'$, then $\interior{\bar h'}$, being an open cube 
or $0$--cube, would have to lie in some such $c'$, a contradiction.  This 
completes the proof of assertion~\eqref{item:proper_containment_restriction}. 

\textbf{Assertion~\eqref{item:l_or_equal}:}  Each part of the assertion was proved above.

\textbf{The deformation retraction:}  If $k=0$, then there is nothing to prove, 
so suppose that $k\geq 1$, and thus, by the above discussion, $k=1$.  It 
follows from the above discussion that $\mathcal 
S=\mathcal S(\bar h)\subset l'$.  We thus have $\mathcal 
S\subset\bigcup_{c'\in\faces'(c)}c'$.  So, 
we just need to exhibit a deformation retraction of 
$\left(\bigcup_{c'\in\faces'(c)}c'\right)\cup f_1$ to 
$\left(\bigcup_{c'\in\faces'(c)}c'\right)$.

Let $\bar h'$ be the cube from 
assertion~\eqref{item:proper_containment_restriction}.  For convenience, let 
$X=\bigcup_{c'\in\faces'(c)}c'$ and let $Y=f_1$.  So, by 
assertion~\eqref{item:proper_containment_restriction}, $Y=\bar h'\times s$.

Clearly $\bar h'\times s$ is the unique maximal cube of $Y$ containing
$\bar h'$.  If $\bar h'$ lies in some other maximal cube of $X\cup Y$,
then $\bar h'\subset X$.  Since $\interior{\bar h'}\cap X=\emptyset$,
this is impossible.  Hence $\bar h'$ is contained in a unique maximal
cube of $X\cup Y$, i.e. $\bar h'$ is a free face of $X\cup Y$. We will
perform the standard collapse of this free face to get a deformation
retraction onto the space obtained by deleting all open cells whose
closures contain the interior of $\bar h'$.

Let $y$ be a (closed) cube of $Y$ containing $\bar h'$.  Then $y$ cannot be 
contained in a cube of $X$, because such a cube would contain $\interior{\bar 
h'}\subset y$.  Hence the unique maximal cube of $X\cup Y$ containing $y$ is 
$\bar h'\times s$.

Let $Y_1$ be the subcomplex of $Y$ obtained by removing
$\interior{\bar h'}$ and $\interior{y}$ for every cube $y$ of $Y$ that
contains $\bar h'$.  Passing from $X\cup Y$ to $X\cup Y_1$ amounts to
collapsing the free face $\bar h'$ of $X\cup Y$, so there is a
deformation retraction $X\cup Y\to X\cup Y$ (fixing $X$) of $X\cup Y$
onto $X\cup Y_1$.

We now collapse $X\cup Y_1$ to $X$ as follows.  Let $d$ be a subcube of $\bar 
h'$.  Then $\interior{d}\subset X$ if and 
only if $\interior{d}\times s\subset X$.  Indeed, if $\interior{d}$ lies in a 
codimension--$1$ face $c'$ containing a persistent corner, then $c'$ is 
parallel to some hyperplane $E$ crossing $\bar h'$ (otherwise $\interior{\bar 
h'}$ would lie in $c'$) and hence not crossing $s$.  So, $\interior{d}\times 
s\subset c'$.

Now, for each maximal subcube $d$ of $\bar h'$ that lies in $Y_1$ but
not in $X$, we have that $d$ is a free face of $Y_1$, contained in the
unique maximal cube $d\times s$, and we can now collapse these cubes
to get $X \cup Y_2$.  Repeating this finitely many times we eventually
obtain the subcomplex $X$, as required.
\end{proof}

\subsection{CAT(0)ness and compatible collapse for the  $\mathcal F(c)$}
We now describe how 
to deformation retract $c$ to $\mathcal F(c)$ compatibly with the corresponding 
deformation retractions of the other cubes.

\begin{lem}\label{lem:big_external}
Let $c$ be an external cube for which $\mathcal D(c)$ is connected and let $f$ 
be a completely external proper subcube of $c$.  Then $f$ lies in a 
codimension--$1$ external subcube of $c$.
\end{lem}

\begin{proof}
If $c$ is completely external, the claim is obvious, so suppose there is an 
internal $1$--cube.  If $f$ is a codimension--$1$ subcube, then we're done.  

As before, let $h=h(c)$ be persistent subcube of 
$c$, and let $\bar h=\bar h(c)$ be salient subcube of $c$.  (It is possible 
that $h=\bar h$.)

Suppose that $f$ is a maximal completely external cube and has codimension at 
least $2$, and that $f$ does not lie 
in any codimension--$1$ external sub-cube.  By Lemma~\ref{lem:extra_cubes}, we 
have that $\bar h\subset f$ and $h\cup\bar 
h\subset l$ for some external codimension--$1$ face $l$ containing a persistent 
corner.  

Let $l$ be as above. 
 Since $\bar h\subseteq f$ is completely external, $\bar 
h\subset\mathcal D(l)$.  If $f\subseteq l$, we have contradicted that $f$ does 
not lie in any codimension--$1$ external face.  Hence, since $f\cap 
l\neq\emptyset$ (it contains $\bar h$), we have that $f\cap l$ is a 
codimension--$1$ face of $f$.  

Now, since $\mathcal D(c)$ is connected and $l$ is external, 
Lemma~\ref{lem:connected_hereditary} implies that $\mathcal D(l)$ is connected.

Let $D$ be the unique hyperplane of $c$ not crossing $l$.  Since $f$ intersects 
$l$ but does not lie in $l$, we have that $f$ intersects $D$.  Now, since $D$ 
does not separate $h,\bar h$, and does not cross $h$ or $\bar h$, every 
$1$--cube dual to $D$ is external.  By induction on dimension and connectedness 
of $\mathcal D(l)$, the completely external cube $f\cap l$ is contained in an 
external codimension--$1$ face $t$ of $l$; recall that $t$ contains a 
persistent corner $v$ of $l$.  Let $e$ be the $1$--cube of $c$ dual to $D$ and 
emanating from $v$.  Then $e$ is external, so $t\times e$ is a codimension--$1$ 
face of $c$ with a persistent corner, $v$, and $f=(f\cap l)\times e$ lies in 
$t\times e$.
\end{proof}

\begin{lem}\label{lem:external_crush}
Let $c$ be an external cube. Then there is a strong deformation 
retraction 
$\Delta_c:c\times[0,1]\to c$ so that $\Delta_c(-,0)$ is the identity and 
$\Delta_c(-,1)$ is a retraction $c\to\mathcal F(c)$.  

Moreover, if $c'$ is a sub-cube of $c$ which is external and has codimension $\ell$, then the 
restriction of $\Delta_c$ to $c'\times[0,1]$ coincides with the identity on $c'\times[0,1-\frac{1}{2^{\ell}}]$ and, on 
$c'\times[1-\frac{1}{2^{\ell}},1]$, restricts to the map given by $(x,t)\mapsto \Delta_{c'}(x,\frac{t+2^\ell-1}{2^\ell})$.
\end{lem}

\begin{proof}
Let $\faces'(c),h=h(c),\bar h=\bar h(c),\{\mathcal S(w)\}$ and 
$\mathcal 
F(c)$ be as 
in the definition of $\mathcal F(c)$.

We will argue by induction on $d=\dimension c$.  In the base case, $d\leq 1$, 
then $c$ is external only if $c$ is completely external, in which case 
$c=\mathcal F(c)$ and we are done.  More generally, whenever $c$ is completely 
external, we take $\Delta_c:c\times[0,1]\to c$ to be projection onto the $c$ 
factor.  Lemma~\ref{lem:compat} implies that $\mathcal F(c')=c'$ for each face 
$c'$ of $c$, whence the ``moreover'' statement also holds.

Hence suppose that the lemma holds for external cubes of dimension $\leq d-1$. 
The inductive step 
has two parts.

\textbf{The first collapse:} Suppose that $c$ is external but not completely 
external (since the completely external case was handled above).

Recall that $\faces'(c)$ denotes the set of external codimension--$1$ faces 
$c'$ of $c$ such that $c$ contains a persistent corner.  By Lemma~\ref{lem:compat}, 
$\faces'(c)$ is exactly the set of all codimension--$1$ external sub-cubes of 
$c$.  Let $\mathcal G'(c)$ be the union of the elements of
$\faces'(c)$.

Now, $\faces'(c)$ is a proper subset of the set of codimension--$1$ faces of 
$c$.  Otherwise, Lemma~\ref{lem:how_panels_intersect_cubes} would imply that $c$ 
is completely external, a contradiction.  On the other hand, $\mathcal G'(c)$ is 
connected, since $\faces'(c)$ contains a set $c_1',\ldots,c'_d$ of 
pairwise-intersecting codimension--$1$ sub-cubes with one in each parallelism 
class.  Note that $\bigcup_{i=1}^dc'_i$ is contractible.

For each $i\leq d$, let $\bar c_i'$ be the codimension--$1$ face parallel to 
$c_i'$.  Let $I$ be the set of $i\leq d$ for which $\bar c_i'\in\faces'(c)$.  
Without loss of generality, $I=\{1,\ldots,\ell\}$ for some $\ell< d$.  Hence 
$\mathcal G'(c)$ is obtained from $c$ by removing the interior of $c$ together 
with each open cube that is not contained in 
$\bigcup_{i=1}^dc_i'\cup\bigcup_{i=1}^\ell\bar c_i'$.  The cubes $\bar 
c_i,i>\ell$ all contain $\bar h$, and $\mathcal G'(c)$ is obtained by removing 
$\interior{c}$ together with a nonempty, connected subspace of the star 
of $\bar h$ which is obtained from a subcomplex of that star by removing its 
boundary.  Hence $\mathcal G'(c)$ is contractible.

Let $\mathcal S(c)$ be the following subspace of $c$.  First, if $\mathcal 
D(c)$ 
is connected, let $\mathcal S(c)=\emptyset$.  Otherwise, let $h,\bar h\subset 
c$ be 
the sub-cubes defined above, so that $\mathcal F_1(c)$ is obtained by adding 
$\mathcal S(w)$ to $\mathcal F_0(c)$ for each completely external sub-cube $w$ of $\bar 
h$.  Let $S_1,\ldots,S_t$ be the set of all such $\mathcal S(w)$, and let $\mathcal 
S(c)=\bigcup_{i=0}^tS_i$.  Let $\mathcal G(c)=\mathcal G'(c)\cup\mathcal S(c)$.

We claim that $\mathcal G(c)$ is contractible.  First, note that $\mathcal 
S(c)\cup h$ is contractible and that $h\subset \mathcal G'(c)$.  To show that 
$\mathcal G(c)=\mathcal G'(c)\cup\mathcal S(c)$ is contractible, it thus 
suffices 
to show that $\mathcal G'(c)\cap(\mathcal S(c)\cup h)$ is 
contractible.

For each $c'\in\faces'(c)$, let $\mathcal V(c')$ be the (nonempty) set of 
$0$--cubes, each of whose incident $1$--cubes in $c'$ is external, and let 
$h(c')$ be the convex hull of the elements of $\mathcal V(c')$.  In other words, $h(c')$ is the persistent subcube of $c'$ 
and $\mathcal V(c')$ is the $0$--skeleton of $h(c')$, by Lemma~\ref{lema:all_persistent}.

Let $\bar 
h(c')$ be the cube of $c'$ that is parallel to $h(c')$ and separated from $h(c')$ by exactly those hyperplanes of $c'$ that 
do not cross $h(c')$ but cross external $1$--cubes.  Note that $h\cap c'\subseteq h(c')$.  
Indeed, each $0$--cube of $c'$, all of whose incident $1$--cubes in $c$ is 
external, certainly belongs to $\mathcal V(c')$, and the convex hull of all such 
$0$--cubes is $h\cap c'$.

\begin{enumerate}
 \item For any codimension--$1$ face $c'$ of $c$, if $h\subset c'$, then $c'$ is 
external.  Moreover, either $(\mathcal S(c)\cup h)\cap c'=h$ or
$\mathcal S(c)\cup h\subset 
c'$.  Indeed, if the hyperplane of $c$ not crossing $c'$ does not abut a panel, 
then the latter holds because $h$ and $\bar h$ are both on the same
side of this hyperplane in $c'$; otherwise, the former holds.
 
 \item If $c'$ is a codimension--$1$ face and $h\cap c'\neq\emptyset$ but 
$h\not\subset c'$, then again $c'$ is external, and $(\mathcal S(c)\cup h)\cap 
c'=\mathcal S(c)\cap c'$.  For each completely external cube $w$ of $\bar h$, we have that $\mathcal S(w)\cap c'$ is a cube 
in $\mathcal S(c')$, where $\mathcal S(w)\cap c'$ is regarded as a cube (using its product structure).
 
 \item If $\bar h\subseteq c'$, then either $h\subset c'$ or $h\cap 
c'=\emptyset$, in which case $c'$ is not external.  Indeed, if the hyperplane 
$H$ parallel to $c'$ separates $h,\bar h$, then it separates $c'$ from $h$, in 
which case $c'$ is not external since external faces contain persistent 
corners.
\end{enumerate}

From the first two statements, any codimension--$1$ face $c'$ of $c$ 
intersecting $\mathcal S(c)\cup h$ is either external or disjoint from $h$.  In 
the 
latter case, $c'$ must contain $\bar h$, so the third statement implies 
that $c'\not\in\faces'(c)$.  It follows that 
$(\mathcal S(c)\cup h)\cap\mathcal G'(c)$ is contractible: just collapse $h$ to 
a 
point and observe that $(\mathcal S(c)\cup h)\cap\mathcal G'(c)$ becomes the 
cone 
on the union of some of the completely external cubes of $\bar 
h$.

Now, $\mathcal G(c)$ is 
a contractible subcomplex of a finite subdivision of $c$, which is a 
contractible CW 
complex.  Let $\mathcal G''(c)$ consist of $\mathcal G(c)$, together with 
any completely external subcube $f$ of $c$ that is not contained in any element 
of $\faces'(c)$.  By Lemma~\ref{lem:extra_cubes}, $\mathcal G(c)$ is a strong 
deformation retract of $\mathcal G''(c)$, so $\mathcal G''(c)$ is 
contractible.

Hence $\mathcal G''(c)$ is a strong deformation retract of $c$, by 
Whitehead's theorem~\cite{Whitehead:i,Whitehead:ii}.  Let 
$\omega_c:c\times[0,1]\to c$ be a strong deformation retraction from $c$ to 
$\mathcal G''(c)$.

\textbf{The second collapse:} By induction on $d$, for each $c'\in\faces'(c)$, 
we have  
strong deformation retractions $\Delta_{c'}:c'\times[0,1]\to c'$ with 
$\Delta_{c'}(-,1)$ a retraction $c'\to\mathcal F(c')$.  

Let $c',c''\in\faces'(c)$.  Then $c'\cap c''$ is an external proper sub-cube of $c'$ and $c''$, so by induction on 
dimension, we have $$\Delta_{c'}|_{(c'\cap c'')\times\{1\}}=\Delta_{c''}|_{(c'\cap c'')\times\{1\}}=\Delta_{c'\cap 
c''}(-,1).$$

Pasting thus provides a deformation 
retraction $\alpha_c:\mathcal G'(c)\times[0,1]\to \mathcal G'(c)$ so that 
$\alpha_c(-,1)$ is a retraction $\mathcal 
G'(c)\to\bigcup_{c'\in\faces'(c)}\mathcal F(c')$.  We now show how to extend 
$\alpha_c$ from a deformation retraction collapsing $\mathcal G(c)$ to one 
collapsing $\mathcal G''(c)$.

If $\mathcal D(c)$ is connected (so that $\mathcal F(c)=\mathcal D(c)$), then 
by 
Lemma~\ref{lem:connected_hereditary}, we have $\mathcal F(c')=\mathcal 
D(c')=c'\cap\mathcal F(c)$ for each $c'\in\faces'(c)$.  We also 
have $\bigcup_{c'\in\faces'(c)}\mathcal F(c')=\bigcup_{c'\in\faces'(c)}\mathcal 
D(c')=\mathcal D(c)=\mathcal F(c)$.  Indeed, if $f$ is a completely external 
cube of $c$, then $f$ is contained in some 
codimension--$1$ external cube, by Lemma~\ref{lem:big_external}.  Hence, in the 
connected case, $\alpha_c$ is a deformation retraction of $\mathcal 
G(c)=\mathcal G'(c)=\mathcal G''(c)$ onto $\mathcal F(c)$.

If $\mathcal D(c)$ is disconnected, then $\mathcal S(c)\neq\emptyset$.  Indeed, the construction of $\mathcal F(c)$ ensures 
that $\mathcal F(c)$ contains some $\mathcal S(w)$ whenever $\mathcal D(c)$ is disconnected (Remark~\ref{rem:nonempty_S}) 
and $\mathcal S(w)$ is contained in $\mathcal S(c)$ by definition.

Now, for 
each $S_i$, and each $c'\in\faces'(c)$, we have 
$S_i\cap c'\subseteq\mathcal F(c')$, by Lemma~\ref{lem:overlap}.  Hence we can 
extend the deformation retraction $\alpha_c$ to a 
strong deformation retraction $\alpha_c:[\mathcal G'(c)\cup\mathcal 
S(c)]\times[0,1]=\mathcal G(c)\times[0,1]\to\mathcal 
G(c)$ with $\alpha_c(-,1)$ a retraction $\mathcal G(c)\to\mathcal F_1(c)$.  
Indeed, we just extend each $\alpha_c(-,t)$ over each $S_i\cap(c-\mathcal 
G'(c))$  by declaring it to be 
the identity.  Lemma~\ref{lem:extra_cubes} ensures that $\mathcal G''(c)$ 
differs from $\mathcal G'(c)$ by adding a (possibly empty) set of cubes $f$ 
that intersect $\mathcal G'(c)$ in cubes of the various $\mathcal 
F(c'),c'\in\faces(c)$.  Hence we can extend $\alpha_c$ to a strong deformation 
retraction $\alpha_c:\mathcal G''(c)\times[0,1]\to\mathcal 
G''(c)$ by declaring it to be the identity on each new cube $f$; this gives the 
desired deformation retraction from $\mathcal G''(c)$ onto $\mathcal 
F(c)$.

\textbf{The composition:}  Now let 
$\Delta_c:c\times[0,1]\to c$ be given by $$\Delta_c(x,t)=\begin{cases}
                                                              \omega_c(x,2t), 
&\ t\in[0,\half]\\
\alpha_c(\omega_c(x,1),2(t-\half)), &\ t\in[\half,1].
                                                         \end{cases}$$
This gives the desired strong deformation retraction.

\textbf{Compatibility:}  Let $c'\subset c$ be an external sub-cube of $c$.  To 
prove the ``moreover'' statement, it suffices to handle the case where $c'$ is 
codimension--$1$; the other cases follow by induction on dimension.  We saw 
above that externality of $c'$ implies that $c'\in\faces'(c)$.  Hence, by 
definition, $\Delta_c$ restricts on $c'\times[0,\frac12]$ to the identity and on $c'\times[\frac12,1]$ to the 
map given by $(x,t)\mapsto\Delta_{c'}(x,t-\frac12)$, as 
required.
\end{proof}

\begin{cor}\label{cor:simply_connected}
For each external cube $c$, the subspace $\mathcal F(c)$ is contractible.
\end{cor}

\begin{proof}
By Lemma~\ref{lem:external_crush}, $\mathcal F(c)$ is a deformation retract 
of $c$ and is thus contractible.
\end{proof}

Our goal is to prove that replacing each external cube $c$ by $\mathcal F(c)$ 
creates a new CAT(0) cube complex.  First, we need:

\begin{lem}\label{lem:cube_complex_link_condition}
Let $c$ be an external cube.  Then $\mathcal F(c)$ is a  cube 
complex whose hyperplanes are components of intersection of $\mathcal F(c)$ with 
hyperplanes of $c$.
\end{lem}

\begin{proof} If $c$ is completely external, then the conclusion is
  obvious, we therefore assume that $\mathcal F(c) \neq c$.  The cubes
  of $\mathcal F(c)$ are of two types: sub-cubes of $c$ belonging to
  $\mathcal F(c)$, or cubes of the form $\mathcal S(w)$, where $w$ is
  a completely external cube of $c$ (recall definition of
  $\mathcal F_1(c)$).  To see that this set of cubes makes
  $\mathcal F(c)$ a cube complex, we must check that if $d,d'$ are
  cubes of the above two types, then $d\cap d'$ is a cube of one of
  the above two types.

For each $d$ as above, we define an associated sub-cube $f$ of $c$ as follows.  
If $d$ is a sub-cube of $c$, then $d$ 
is completely external, by Lemma~\ref{lem:no_interior}.  In this case, define $f=d$.  
Otherwise, 
if $d=\mathcal S(w)$ for some $w$, let $f$ be the smallest sub-cube of $c$ containing 
the two 
parallel copies of $w$ of whose union $\mathcal S(w)$ is the $\ell_2$--convex hull.  
Again, $f$ is 
external because, by the definition of $\mathcal F(c)$, the subspace $\mathcal S(w)$ 
contains a persistent corner of $c$.

Define $f'$ analogously for $d'$.  Then $d$ is a cube of
$\mathcal F(f)$ and $d'$ is a cube of $\mathcal F(f')$, by
Lemma~\ref{lem:compat}. First consider the case where $f=f'$:

\begin{claim}\label{claim:equal_f}
If $f=f'$, then $d=d'$.     
\end{claim}

\renewcommand{\qedsymbol}{$\blacksquare$}
\begin{proof}[Proof of Claim~\ref{claim:equal_f}]
  If $d,d'$ are both sub-cubes of $c$, this is clear.  Otherwise,
  suppose $d=\mathcal S(w)$.  This means that, letting
  $h,\bar h\subset c$ be the persistent, salient subcubes of $c$
  respectively, $w$ is a completely external cube of $\bar h$,
  $\bar w\subset h$ is the parallel copy of $w$ separated from $w$ by
  exactly those abutting hyperplanes of $c$ not crossing $h$, and
  $\mathcal S(w)$ is the $\ell_2$--convex hull of $w\cup\bar w$, and
  $f$ is the $\ell_1$--convex hull of $w\cup\bar w$.  Now,
  $f\cap h=\bar w$ and $f\cap \bar h=w$. Indeed, on the one hand, the
  containment $\bar w \subset f\cap h$ is obvious. On the other hand
  if $f \cap h \supsetneq \bar w$, then the intersection must contain a 1-cube $e$ not in
  $\bar w$ but in $h$. By definition $e$ is dual to a hyperplane not
  separating $w,\bar w$, which means it can't be on an $\ell_1$
  geodesic connecting $w$ and $\bar w$.

  Thus, any ``diagonal'' cube $\mathcal S(w')$ intersecting $f$ has
  $\mathcal S(w')\cap f\subseteq \mathcal S(w)$.  So, if
  $d'=\mathcal S(w')$ for some $w'$, then $d'\subseteq d$ and, by
  symmetry, $d\subseteq d'$, i.e. $d=d'$.

  The remaining case is where $f=f'=d'$ and $d=\mathcal S(w)$. By
  Lemma \ref{lem:overlap} $\mathcal F(f) = f \cap \mathcal
  F(c)$, but since $d'$ is assumed to be completely external $\mathcal
  F(f)=f$. We now have that  $w\cup \bar w$ are not contained in a
  codimension--1 face of $f$. Since $\bar w$ is a subcube of the
  persistent subcube $h(f)$ and $w$ is a subcube of the
  salient subcube $\bar h(f)$, it follows by Lemma
  \ref{lem:extra_cubes}, assertion~\eqref{item:l_or_equal},
  that $\mathcal D(f)$ is disconnected -- contradicting that $f$ is
  completely external.
\end{proof}

Next, consider the case where $f\neq f'$.

\begin{claim}\label{claim:restrict_cubes}
Let $c$ be an external cube.  Let $c''$ be an external sub-cube of 
$c$.  Let $d$ be a cube of $\mathcal F(c)$.  Then $d\cap c''$ is either empty 
or a cube of $\mathcal F(c'')$.  Hence $d\cap\mathcal F(c'')$ is either empty or a cube of $\mathcal F(c'')$.     
\end{claim}

\begin{proof}[Proof of Claim~\ref{claim:restrict_cubes}]
Suppose that 
$d\cap c''\neq\emptyset$.  Then $\mathcal F(c'')=\mathcal F(c)\cap c''$ by 
Lemma~\ref{lem:compat}.  Hence 
the cubes of $\mathcal F(c'')$ are the cubes of $\mathcal F(c)$ that lie in 
$c''$.  Thus it suffices to check that $d\cap c''$ is a subcube of $d$.

If $d$ is completely external, then $d\cap c''$ is a completely
external cube of $c''$, and is thus a cube of
$\mathcal D(c'')\subset \mathcal F(c'')$, as required.

Next, suppose that $d=\mathcal S(w)$, where $w$ is some completely
external cube of some salient cube $\bar h$.  Let $\bar w$ be the cube
of the corresponding persistent cube $h$ with the property that
$\mathcal S(w)$ is the $\ell_2$--convex hull of $w\cup \bar w$.  If
$w,\bar w$ both intersect $c''$, then
$\mathcal S(w\cap c'')=\mathcal S(w)\cap c''=d\cap c''$ is a subcube
of $d$.  Otherwise, $c''$ has nonempty intersection with only one of
$w,\bar w$. Without loss of generality $c'' \cap w \neq \emptyset$ and
$c'' \cap \bar w = \emptyset$. Since $c''$ is a subcube of $c$, by
definition of $\mathcal S(w)$, we have
$\mathcal S(w)\cap c''=w\cap c''$, which is again a subcube of $d$.
This proves the claim.
\end{proof}

\begin{claim}\label{claim:cube_intersection}
Let $c,c'$ be distinct external cubes, neither of which contains the other.  Let $d,d'$ be cubes of $\mathcal F(c),\mathcal 
F(c')$ respectively.  Then $d\cap d'$ is either empty or is a cube of $\mathcal F(c)$ and $\mathcal 
F(c')$.     
\end{claim}

\begin{proof}[Proof of Claim~\ref{claim:cube_intersection}]
  Let $c''=c\cap c'$, which (if nonempty) is an external sub-cube of
  $c$ and $c'$ by Lemma \ref{lem:external_cube_intersections}.  By
  Lemma~\ref{lem:compat},
  $\mathcal F(c)\cap\mathcal F(c')=\mathcal F(c'')$.  Hence
  $d\cap d'=(d\cap\mathcal F(c''))\cap(d'\cap \mathcal F(c''))$.  By
  Claim~\ref{claim:restrict_cubes}, $d\cap\mathcal F(c'')$ and
  $d'\cap \mathcal F(c'')$ are subcubes of $\mathcal F(c'')$.  Let
  $t,t'$ be the subcubes of $c''$ that are the $\ell_1$--convex hulls
  of $d\cap \mathcal F(c''),d'\cap \mathcal F(c'')$ respectively.
  Note that $t,t'$ are external.  So, by
  Claim~\ref{claim:restrict_cubes}, $d\cap\mathcal F(c'')$ and
  $d'\cap\mathcal F(c'')$ are cubes of $\mathcal F(t),\mathcal F(t')$
  respectively.  Note that $t,t'$ are proper subcubes of $c,c'$ (and
  therefore of lower dimension) since they are contained in $c''$,
  which is a proper subcube of $c$ and $c'$ by hypothesis.

  So, if $t\neq t'$, then by induction on dimension,
  $(d\cap\mathcal F(c''))\cap(d'\cap \mathcal F(c''))$ is a cube of
  $\mathcal F(t)\cap \mathcal F(t')$ (and hence of
  $\mathcal F(c)\cap \mathcal F(c')$).  If $t=t'$, then it follows
  from Claim~\ref{claim:equal_f} above that
  $d\cap\mathcal F(c'')=d'\cap \mathcal F(c'')$.
\end{proof}
\renewcommand{\qedsymbol}{$\Box$}

Recall that $d,d'$ are cubes of $\mathcal F(c)$, contained respectively in external cubes $f,f'$ of $c$ defined 
above.  Moreover, we are considering the case where $f\neq f'$.  Suppose that neither $f$ nor $f'$ contains the other.  Now, 
$d$ is a cube of $\mathcal F(f)=\mathcal F(c)\cap f$ and $d'$ is a cube of $\mathcal F(f')=\mathcal F(c)\cap f'$, by 
Claim~\ref{claim:restrict_cubes}.  Hence $d\cap d'$ is either empty or a cube of $\mathcal F(f')$, and hence of $\mathcal 
F(c)$, by Claim~\ref{claim:cube_intersection}.  

Finally, consider the case where $f\subsetneq f'$.  In this case, $d$
is a cube of $\mathcal F(f)$, by Claim~\ref{claim:restrict_cubes}.
Moreover, $d'\cap f$ is either empty or a cube of $\mathcal F(f)$, by
Claim~\ref{claim:restrict_cubes}.  Since $d\cap d'=d\cap d'\cap f$,
either $d\cap d'=\emptyset$, or $d\cap d'$ is the intersection of the
cubes $d\cap f$ and $d'\cap f$ of $f$.  Since $f$ is a proper face of
$c$, induction on dimension yields that $d\cap d'$ is a cube of
$\mathcal F(f)$ and hence of $\mathcal F(c)$.

This completes the proof that $\mathcal F(c)$ is a cube complex. 

\textbf{Hyperplanes:} We have shown that $\mathcal F(c)$ is a 
cube complex.  The statement about hyperplanes is immediate from the definition 
of the cubes of $\mathcal F(c)$.
\end{proof}

\begin{lem}\label{lem:cube_intersection}
Let $c$ be an external cube.  Then $\mathcal F(c)$ is a CAT(0) cube complex.
\end{lem}

\begin{proof}
By Lemma~\ref{lem:cube_complex_link_condition}, $\mathcal F(c)$ is a cube complex.  By Corollary~\ref{cor:simply_connected}, 
$\mathcal F(c)$ is simply connected.  So, in view of~\cite[Theorem 5.20]{ bridson2011metric} we just have to check that the 
given cubical structure on $\mathcal F(c)$ has the property that each vertex-link is a flag complex.

Let $v$ be a vertex of $c$.  Consider a clique in the link of $v$ in
$\mathcal F(c)$. Then $v$ is contained in $e_1,\ldots,e_k$, which are
(necessarily external) $1$--cubes of $\mathcal F(c)$ that are
$1$--cubes of $c$, and $v$ is contained in $f_1,\ldots,f_r$, which are
$1$--cubes of $\mathcal F(c)$ that are ``diagonal'' (i.e. $1$--cubes
of $\mathcal F(c)$ but not of $c$).  Since we are considering a
clique, for all $i,j$, we have that $e_i,e_j$ span a $2$--cube of
$\mathcal F(c)$, and the same is true for $f_i,f_j$ and $e_i,f_j$.

Since $c$ is CAT(0), the $1$--cubes $e_1,\ldots,e_k$ span a 
$k$--dimensional sub-cube $c'$ of $c$ that contains a collection of 
$k\choose{2}$ intersecting completely external $2$--cubes $s_{ij}$, with 
$s_{ij}$ spanned by $e_i,e_j$.  Then for each $\ell\le k$, the union of the 
$s_{ij}$ contains $k$ $1$--cubes parallel to $e_\ell$, all of which are 
external.  Let $f$ be a codimension--$1$ face of $c'$ containing $1$--cubes 
parallel to $e_\ell$.  For each hyperplane $H$ crossing $c'$ which is not dual to $e_\ell$, there is a 
$1$--cube $e'_\ell(H)$ parallel to $e_\ell$ and separated from $e_\ell$ by $H$. 
 In particular, $f$ contains an external $1$--cube parallel to $e_\ell$.  It 
follows from Lemma~\ref{lem:how_panels_intersect_cubes} that $c'$ is completely 
external, and hence $c'\subset\mathcal F(c)$.

Next, note that $r\leq 1$.  Indeed, by construction, each $2$--cube of 
$\mathcal F(c)$ that is not a sub-cube of $c$ has the form $d\times I$, where 
$d$ is a $1$--cube of $c$ and $I$ is a diagonal interval.  Hence $f_i,f_j$ 
cannot span a $2$--cube of $\mathcal F(c)$, so $r\leq 1$. 

If $r=0$, the flag condition holds since $c'\subset\mathcal F(c)$. Now
suppose $r=1$.  For each $i$, the $1$--cubes $e_i,f_1$ of
$\mathcal F(c)$ span a (diagonal) $2$--cube $\sigma_i$ of
$\mathcal F(c)$.  Now, by the definition of $\mathcal F(c)$, the cube
$\sigma_i$ is the convex hull of $e_i,\bar e_i$, where $e_i,\bar e_i$
are parallel external $1$--cubes lying in $h(c),\bar h(c)$ or vice
versa.  Indeed, every ``diagonal'' square of $\mathcal F(c)$ is
spanned by a ``diagonal'' $1$--cube and a $1$--cube of $h(c)$.  By
interchanging the names of $e_i,\bar e_i$ if necessary, we can assume
that $e_i\subset h(c)$ for all $i$ and $\bar e_i\subset\bar h(c)$.
Hence, by Lemma~\ref{lema:all_persistent}, both $0$--cubes of $e_i$
are persistent corners.  It follows that $c'\subseteq h(c)$.  Indeed,
$c'$ is spanned by $1$--cubes $e_1,\ldots,e_k$, all of whose
$0$--cubes are persistent corners, and the convex hull of this set of
$0$--cubes is contained in $h(c)$ and, on the other hand, equal to
$c'$.

Let $\bar c'=\prod_{i=1}^k\bar e_i$, so $\bar c'\subseteq \bar h(c)$.  Each 
$0$--cube $v$ of $\bar c'$ is thus a completely 
external cube of $\bar h(c)$, and is thus joined by a diagonal $1$--cube 
(parallel to $f_1$) to a $0$--cube of $c'$.

Now, if $\bar c'$ is completely external, then by construction, $\mathcal F(c)$ 
contains a cube $c'\times f_1=\bar c'\times f_1$ spanned by 
$e_1,\ldots,e_k,f_1$, so the link condition is satisfied.  So, it suffices to 
show that $\bar c'$ is completely external.

To that end, it suffices to show that every $1$--cube $\bar e$ of $\bar c'$ is 
external.  Let $t$ be the smallest subcube of $c$ containing $c'\cup\bar c'$. 
  Let $e$ be the $1$--cube of $c'$ that is parallel to $\bar e$ and separated 
from $\bar e$ by exactly those hyperplanes separating $c',\bar c'$.

There exists $i$ such that $e$ is parallel to $e_i$.  Let $\ddot e_i$ be the 
$1$--cube of $c'$ that is parallel to $e_i$ and separated from $e_i$ by all 
hyperplanes of $c'$ not crossing $e_i$.  Since $c'$ is completely external, the 
$1$--cube $\ddot e_i$ is external.

Now, $\ddot e_i$ is separated from $\bar e_i$ by all hyperplanes of $t$ not 
crossing $\bar e_i$.  Moreover, since $\bar e_i$ lies in $\sigma_i$, it is 
external.  Hence, by Remark~\ref{rem:codim_1_external}, every $1$--cube of $t$ 
parallel to $\bar e_i$ is external.  In particular, $\bar e$ is external, as 
required.  
\end{proof}

\section{Collapsing $\Psi$ along extremal 
panels}\label{sec:cube_def_ret_extremal}

We can now prove the main theorem about panel collapse.

\begin{thm}[Collapsing extremal panels]\label{thm:pasting}
Let $\Psi$ be a finite-dimensional CAT(0) cube complex and let $\mathcal P$ be a 
collection of extremal panels with the no facing panels property.  Then there 
is a CAT(0) cube complex $\Psi_\bullet\subset\Psi$ so that:
\begin{itemize}
 \item there is a deformation retraction $\Psi\defRet\Psi_\bullet$;
 \item for each hyperplane $K$ of $\Psi$, the subspace $K\cap\Psi_\bullet$ is 
the disjoint union of hyperplanes of $\Psi_\bullet$;
 \item each hyperplane of $\Psi_\bullet$ is a component of $\Psi\cap K$ for some 
hyperplane $K$ of $\Psi$;
 \item for each $P\in\mathcal P$, the inside of $P$ is disjoint from 
$\Psi_\bullet$.
\end{itemize}
\end{thm}

\begin{proof}
Let $\Psi_\bullet$ be the union over all maximal cubes $c$ of $\Psi$ of the 
subspaces $\mathcal F(c)$.  

By Lemma~\ref{lem:no_interior}, $\Psi_\bullet$ is disjoint from the inside of 
each panel in $\mathcal P$, i.e. $\Psi_\bullet$ is disjoint from the 
interior of any cube of $\Psi$ that is internal to a panel in $\mathcal P$.

 By Corollary~\ref{cor:simply_connected} and 
Lemma~\ref{lem:cube_complex_link_condition}, 
each $\mathcal F(c)$ is a CAT(0) cube complex whose hyperplanes are components 
of intersection of $\mathcal F(c)$ with hyperplanes of $\Psi$.  The cubes of 
$\mathcal F(c)$ are either sub-cubes of $c$ or cubes of the form 
$w\times\left[-\frac{\sqrt{\kappa}}{2},\frac{\sqrt{\kappa}}{2}\right]$ (where 
$w$ is a sub-cube of $c$), which has 
dimension at most $d-1$.  By Lemma~\ref{lem:compat}, if $c,c'$ are 
intersecting external cubes, then they intersect along a common external face 
$c''$, and $\mathcal F(c)\cap\mathcal F(c')=\mathcal F(c'')$.  (The 
intersection of distinct external cubes is external since panels in $\mathcal 
P$ are extremal.)

Moreover, if 
$d,d'\subset\mathcal F(c),\mathcal F(c')$ are cubes (including diagonal cubes), 
the intersection $d\cap d'$ is either empty or a cube of $\mathcal F(c)$ and 
$\mathcal F(c')$, by Lemma~\ref{lem:cube_intersection}, so $\Psi_\bullet$ is a 
cube complex.  The definition of a hyperplane in a cube complex (a connected 
union of hyperplanes of cubes that intersects each cube in $\emptyset$ or a 
hyperplane of that cube), together with 
the above characterisation of hyperplanes in $\mathcal F(c)$, establishes the 
statements about hyperplanes of $\Psi_\bullet$. 

To see that $\Psi_\bullet$ is a CAT(0) cube complex, we must first verify that 
it is simply connected, which we will do by constructing the claimed 
deformation retraction $\Psi\defRet\Psi_\bullet$.  For each maximal cube $c$, 
any codimension--$1$ face $c'$ that is not external has the property, by 
definition, that $\interior{c}$ is contained in the inside of some panel 
$P\in\mathcal P$.  Since $P$ is an extremal panel, $c'$ is 
contained in a unique maximal cube of $\Psi$, namely $c$.  Hence the 
deformation retraction $c\defRet\mathcal F(c)$ from 
Lemma~\ref{lem:external_crush} induces a deformation retraction 
$\Psi\times[0,1]\to\Psi$ which is the identity outside of $c$ and agrees on $c$ 
with $\Delta_c$.  Lemma~\ref{lem:external_crush} also implies that these 
deformation retractions are compatible with intersections of external (and 
hence maximal) cubes, so the pasting lemma produces the desired deformation 
retraction $\Psi\defRet\Psi_\bullet$.

It remains to check that $\Psi_\bullet$ is locally CAT(0), i.e. that the link 
of each $0$--cube is a flag complex.  Let $v\in\Psi_\bullet$ be a $0$--cube, so 
that $v$ is also a $0$--cube of $\Psi$.  Let $e_1,\ldots,e_k$ be $1$--cubes of 
$\Psi$ incident to $v$ and lying in $\Psi_\bullet$, and let $f_1,\ldots,f_r$ be 
``diagonal'' $1$--cubes of $\Psi$ incident to $v$.  Suppose that for all $i,j$, 
the $1$--cubes $e_i,e_j$, and $e_i,f_j$, and $f_i,f_j$ span a $2$--cube of 
$\Psi_\bullet$.  

For $i\neq j$, since $f_i,f_j$ span a $2$--cube of $\Psi$, they must lie in a 
common external cube, but this was shown in the proof of 
Lemma~\ref{lem:cube_complex_link_condition} to be impossible.  Hence $r\leq1$.

By CAT(0)ness of $\Psi$, the $e_i$ span a $k$--cube $c'$ of $\Psi$, which is 
necessarily external, since it contains a persistent corner.  So 
Lemma~\ref{lem:cube_complex_link_condition} implies that $c'$ is completely 
external (and hence in $\Psi_\bullet$).  So, if $r=0$, then we are done.

If $r=1$, then let $d$ be the $\ell_1$ convex hull of $f_1$, which is a cube.  
Then for any $i\leq k$, the fact that $e_i,f_1$ span a $2$--cube of 
$\Psi_\bullet$ implies that $\Psi$ contains the cube $e_i\times d$.  Hence 
there is a cube $c=c'\times d$ that contains $f_1$ and all of the $e_i$.  But 
$\Psi_\bullet\cap c=\mathcal F(c)$, 
so Lemma~\ref{lem:cube_complex_link_condition} implies that $\mathcal F(c)$, 
and 
hence $\Psi_\bullet$,  contains a cube spanned by $f_1,e_1,\ldots,e_k$.  This 
completes the proof that $\Psi_\bullet$ is CAT(0).
\end{proof}

\subsection{Equivariant panel collapse}\label{sec:equivariant_panel_collapse}
Let $\Psi$ be a finite-dimensional CAT(0) cube complex and let 
$G\leq\Aut(\Psi)$ act cocompactly.  (We do not assume that $\Psi$ is locally 
finite or that the $G$--action is proper.) We also require that $G$ acts on 
$\Psi$ without inversions in hyperplanes, i.e. for each hyperplane $H$, the 
stabiliser in $G$ of $H$ stabilises both associated 
halfspaces.  This can always be assumed to hold by passing to the first cubical 
subdivision (see e.g.~\cite[Lemma 4.2]{Haglund:semisimple}).

\begin{defn}[Complexity]\label{defn:complexity}
The \emph{complexity} $\#(\Psi)$ of $\Psi$ is the tuple 
$(\#_i(\Psi))_{i=0}^{\dimension\Psi-2}$, where $\#_i(\Psi)$ is the number of 
$G$--orbits of $(\dimension\Psi-i)$--cubes in $\Psi$.  Observe that $\Psi$ is a 
tree if and only if $\#(\Psi)=\vec 0$.  Complexity is ordered lexicographically.
\end{defn}

\begin{cor}\label{cor:lower_complexity}
Suppose that $\Psi$ contains a hyperplane $H$ and a hyperplane $E$ so that 
$E\cap H$ is extremal in $H$.  Then $\Psi$ contains a $G$--cocompact subspace $\Psi_\bullet$ such that:
\begin{itemize}
 \item $\Psi_\bullet$ is a CAT(0) cube complex;
 \item for each hyperplane $K$ of $\Psi$, the subspace $K\cap\Psi_\bullet$ is the disjoint union of hyperplanes of $\Psi_\bullet$;
 \item each hyperplane of $\Psi_\bullet$ is a component of $\Psi_\bullet\cap K$ 
for some hyperplane $K$ of $\Psi$;
 \item $\#(\Psi_\bullet)<\#(\Psi)$ in the lexicographic order on $\naturals^{\dimension\Psi-2}$;
 \item the action of $G$ on $\Psi_\bullet$ is without inversions;
 \item each $g\in G$ is a hyperbolic [resp. elliptic] isometry of 
$\Psi_\bullet$ if and only if it is a hyperbolic [resp. elliptic] isometry of 
$\Psi$.
\end{itemize}
This holds in particular if $G$ is a group acting cocompactly and without 
inversions on a CAT(0) cube complex $\Psi$ under either of 
the following conditions:
\begin{itemize}
 \item some hyperplane of $\Psi$ is compact and contains more than one point;
 \item $\Psi$ contains a hyperplane $H$ so 
that $\stabilizer_G(H)$ does not act essentially on $H$.
\end{itemize}
Finally if $\Psi$ is locally finite then so is $\Psi_\bullet$.
\end{cor}

\begin{proof}
The hypotheses provide an extremal panel $P$; let $\mathcal P=G\cdot P$.  Since 
the action is without inversions, Lemma~\ref{lem:block_intersection_condition} 
implies that $\mathcal P$ has the no facing panels property.  Since 
$\mathcal P$ is $G$--invariant, for each cube $c$ and each $g\in G$, 
Lemma~\ref{lem:canonical} implies $\mathcal F(gc)=g\mathcal F(c)$, so the 
CAT(0) cube complex $\Psi_\bullet$ from 
Theorem~\ref{thm:pasting} is a $G$--invariant subspace of $\Psi$.  
Theorem~\ref{thm:pasting} also says that the hyperplanes of $\Psi_\bullet$ 
have the desired form.  

Since $G$ acts without inversions on 
$\Psi$ and hyperplanes of $\Psi_\bullet$ extend to hyperplanes of $\Psi$, the 
action of $G$ on $\Psi_\bullet$ is without inversions.  Indeed, let $C$ be a 
hyperplane of $\Psi_\bullet$.  Then $C$ is a component of the intersection of 
$\Psi_\bullet$ with pairwise-crossing hyperplanes $H_1,\ldots,H_k$ of $\Psi$, 
for some $k\leq\dimension\Psi$.  Hence $\stabilizer_G(C)$ permutes the 
hyperplanes $H_k$.   

Let $\mathcal B=\bigcap_{i=1}^k\neb(H_i)$, and observe that $\stabilizer_G(C)$ 
acts on $\mathcal B$, preserving the collection of extremal panels from 
$\mathcal P$ in $\mathcal B$.  The no facing panels property implies that these 
panels have nonempty intersection, which must be preserved by 
$\stabilizer_G(C)$.  Hence $\stabilizer_G(C)$ stabilizes $\bigcap_i\OR H_i$, 
where $\OR H_i$ is the halfspace of $\mathcal B$ associated to $H_i$ that 
contains a panel in $\mathcal P$ extremalised by $H_i$.  Since $\bigcap_i\OR 
H_i$ is one of the halfspaces of $\mathcal B\cap\Psi_\bullet$ associated to 
$C$, we see that no element of $\stabilizer_G(C)$ acts as an inversion across 
$C$.

Since there are finitely 
many $G$--orbits of cubes in $\Psi$, and each contributes a bounded number of 
cubes to $\Psi_\bullet$, the action of $G$ on $\Psi_\bullet$ is cocompact.  
Finally, let $c$ be a maximal cube of $\Psi$ that intersects the inside of 
some panel in $\mathcal P$.  Then Theorem~\ref{thm:pasting} implies that 
$\mathcal F(c)$ has dimension strictly lower than that of $c$, so 
$\#(\Psi_\bullet)<\#(\Psi)$.  Finally, suppose that $g\in G$.  
By~\cite[Theorem~1.4]{Haglund:semisimple}, since $\Psi$ is finite-dimensional, 
either $g$ 
is hyperbolic (in the $\ell_1$ and $\ell_2$ metrics) or $g$ fixes a $0$--cube 
(since the action is without inversions).  Now, each $0$--cube of $\Psi$ lies 
in $\Psi_\bullet$, by construction, and the inclusion $\Psi_\bullet\to\Psi$ is 
$G$--equivariant, so stabilisers of $0$--cubes in $\Psi$ do not change on 
passing to $\Psi_\bullet$.  Hence $g$ is hyperbolic or elliptic on 
$\Psi_\bullet$ according only to whether it was hyperbolic or elliptic on 
$\Psi$.  This proves the first part of the claim.

Now, let $G$ act cocompactly and without inversions on a CAT(0) cube complex 
$\Psi$.  If $H$ is a hyperplane of $\Psi$ that is compact 
and contains more than one point, then Corollary~\ref{cor:compact_hyp} shows 
that $\Psi$ has an extremal panel, and we can argue as above.  

Similarly, 
suppose that $\Psi$ contains a hyperplane $H$ so that $\stabilizer_G(H)$ does 
not act on $H$ essentially.  Then there is a hyperplane $E$ bounding a halfspace 
$\vec E\cap H$ in $H$ so that $\vec E\cap H$ and $(\Psi-\vec E)\cap H$ are both 
$\stabilizer_G(H)$--shallow.  

Since $G$ acts cocompactly on $\Psi$, the 
stabilizer of $H$ acts coboundedly on $H$.  To verify this, we will show that there are finitely many $\stabilizer_G(H)$--orbits of $1$--cubes of $\Psi$ dual to $H$.  Indeed, since $G$ acts on $\Psi$ with finitely many orbits of $1$--cubes, there are $1$--cubes $e_1,\ldots,e_k$ dual to $H$ so that every $1$--cube dual to $H$ has the form $ge_i$ for some $i\leq k$ and some $g\in G$.  But since $e_i$ is dual to $H$, the $1$--cube $ge_i$ is dual to $gH$.  On the other hand, since $ge_i$ is dual to $H$, we have $gH=H$, i.e. $g\in\stabilizer_G(H)$, as required.  

Hence one of  $\vec E\cap H$ and 
$(\Psi-\vec E)\cap H$ must be contained in a uniform neighbourhood of $E\cap H$.  
Hence $E$ could be chosen so that $E\cap H$ is extremal in $H$.  Thus $\Psi$ has 
an extremal panel, and we can conclude as above.

The final assertion follows since for each cube $c$, $\mathcal F (c)$
has finitely many cubes.
\end{proof}

\begin{rem}[Effect of panel collapse on hyperplanes]\label{rem:hyper_effect}
Let $\Psi$, $G$, and $\mathcal P$ be as in
Corollary~\ref{cor:lower_complexity}. Let $H,E$ be hyperplanes of
$\Psi$ so that $E\cap H$ is extremal in $H$.  Let $P$ be the extremal
panel in $N(E\cap H)$ abutted by $H$ and extremalised by $E$.  Then
$E\cap\Psi_\bullet$ is contained in the closure of $\left[E-(E\cap
  N(E\cap H))\right]\cup (E\cap H)$ (that is to say, $E$ is replaced by the 
halfspaces of $E$ induced by $N(E\cap H)$, except that $E\cap H$ might be 
added back in the event that there were diagonal cubes in $\Psi_\bullet$).  
This can be seen by inspecting the
construction of $\Psi_\bullet$ at the level of individual cubes.
\end{rem}

\section{Applications}\label{sec:examples}

\subsection{Stallings's theorem}\label{sec:main_corollary}
We now prove Theorem~\ref{thm:Stallings}.

\begin{cor}\label{cor:stallings}
Let $G$ be a finitely generated group with more than one end.  Then $G$ splits 
nontrivially as a finite graph of groups with finite edge 
groups.
\end{cor}

\begin{proof}
Fix a locally finite Cayley graph $X$ of $G$.  Since $G$ has more than one end, 
there is a ball $K$ so that $X-K$ has at least two components containing points 
arbitrarily far from $K$.  Applying~\cite{sageev1995ends}, we obtain an action 
of $G$ on a CAT(0) cube complex $\Psi_0$ whose hyperplanes correspond to $K$ 
and its translates; moreover, $G$ has no global fixed point in $\Psi_0$.  For 
any $r\geq0$, there exists $d=d(r)$ so that if $g_1,\ldots,g_k$ are such that 
$\dist_X(g_iK,g_jK)\leq r$ for all $i,j$, then there exists $y\in X$ so that 
each $g_iK$ intersects $B_d(y)$.  It follows that $\Psi_0$ has finitely many 
$G$--orbits of cubes, i.e. $G$ acts cocompactly (and with a single orbit of 
hyperplanes) on $\Psi_0$.   Moreover, the hyperplane-stabilisers are 
commensurable with the 
conjugates of the stabiliser of $K$, so they are finite.  Hence the hyperplanes 
of $\Psi_0$ are compact.   By passing to the first cubical subdivision, we can 
assume that $G$ acts without inversions.  (Up to here, our argument does not 
essentially differ from that in~\cite{niblo2004geometric}; the arguments 
diverge 
at the next step.)

Hence, by Corollary~\ref{cor:lower_complexity}, there is a sequence 
$\ldots\subsetneq\Psi_n\subsetneq\Psi_{n-1}\subsetneq\ldots\subsetneq\Psi_0$ of 
$G$--CAT(0) cube complexes, so that $\#(\Psi_n)<\#(\Psi_{n-1})$ for all 
$n\geq1$ for which $\Psi_n$ is not a tree.  Since $\#(\Psi_0)$ is finite, there 
exists $n$ so that 
$\#(\Psi_n)=\vec 0$, i.e. $\Psi_n$  is a tree.  Moreover, by 
Corollary~\ref{cor:lower_complexity}, the stabiliser of a hyperplane in 
$\Psi_n$ 
(i.e. an edge stabiliser) is virtually contained in the stabiliser of a 
hyperplane of 
$\Psi_0$, and is thus finite.  Since $\Psi_n$ is contained in $\Psi_0$, and $G$ 
did not fix a point in $\Psi_0$, there is no point in $\Psi_n$ fixed by $G$, 
i.e. the splitting is nontrivial. 
\end{proof}

\subsection{Antenna cubulations of free groups}\label{sec:antenna}
In~\cite{Wise:recube}, Wise proved the following remarkable
theorem. Let $F$ be a finite-rank free group, and let $H\leq F$ be an
arbitrary finitely-generated subgroup of infinite index.  Then $F$
acts freely and cocompactly on a CAT(0) cube complex $\Psi$, and
$\Psi$ contains a single $F$--orbit of hyperplanes, with each
hyperplane stabiliser conjugate to $H$.  This gives a profusion of
cubulations of $F$ beyond the obvious $1$--dimensional ones.  In fact,
the hyperplanes of $\Psi$ are necessarily non-compact when $H$ is
infinite, so $\Psi$ can be a bit hard to picture: on one hand, it is
quasi-isometric to a tree; on the other hand, the separating compact
sets one sees cannot be hyperplanes.  To illustrate
Theorem~\ref{thmi:panel_collapse}, we will explain how to perform
panel collapse on these exotic cubulations of $F$.

We say a CAT(0) cube complex $\Psi$ with a free, cocompact $F$--action is an 
\emph{antenna cubulation} of $F$ if it is of the type constructed 
in~\cite{Wise:recube} (which we discuss in more detail below).

\begin{prop}\label{prop:antenna}
  Let $\Psi$ be an antenna cubulation of $F$.  Then $\Psi$ has a
  hyperplane $D$ such that $\stabilizer_D(\Psi)$ does not act
  essentially on $D$, so $\Psi$ has an extremal panel $P$.
\end{prop}

From Corollary~\ref{cor:lower_complexity}, we can then apply panel collapse.  The interesting thing about 
Proposition~\ref{prop:antenna} is that it provides an example of panel collapse 
when the hyperplanes are non-compact but the group is not too complicated.

\begin{proof}[Proof of Proposition~\ref{prop:antenna}]
First we recall the construction of antenna cubulations.  Then we apply 
Theorem~\ref{thmi:panel_collapse} and prove the two statements.

\textbf{Antenna cubulations:}  Let $C$ be a wedge of finitely many oriented 
circles, labelled by the generators of $F$, so that $F$ is identified with 
$\pi_1C$.  Let $\widehat C\to C$ be the based cover corresponding to 
$H\hookrightarrow F$, and let $\bar C$ be the core of $\widehat C$, which is 
compact since $C$ is finitely generated.  Since $[F:H]=\infty$, there is a 
vertex $v$ of $\bar C$ at which the immersion $\bar C\to C$ is not locally 
surjective.

An \emph{antenna} is a tree $A=P\cup\bigcup_iT_i$, where $P\cong [0,n]$ is a 
tree with $n$ edges and exactly two leaves, $T_i$ is a path of length $2$ for 
$1\leq i\leq n$, and $A$ is formed by identifying the midpoint of each $T_i$ to 
the vertex $i$ of $P$.  We also orient the edges of $A$ and label them by the 
generators of $F$ in such a way as to obtain an immersion $A\to C$.  We also do 
this labelling so that:
\begin{itemize}
     \item the edge $[0,1]$ of $P$ is labelled in such a way that attaching $P$ 
to $\bar C$ by identifying $v$ with $0\in P$ yields an immersion 
$P\cup_{0=v}\bar C\to C$ given by the labels and orientations;
    \item every reduced word of length $2$ in the generators of $F$ appears as 
the labelling of some $T_i$ (thus $n$ is bounded below in terms of the rank of 
$F$).  Let $W= P\cup_{0=v}\bar C.$
\end{itemize}

Now add (infinite) trees to $W$ where necessary to obtain an infinite-sheeted 
cover $\ddot C\to C$ homotopy equivalent to $W$.  For each of these finitely 
many trees, assign it a $+$ or $-$, and declare $W$ itself to be $-$.  This 
induces a partition of the universal cover $\widetilde C$ into two halfspaces (points mapping into 
$+$ and points mapping into $-$), and with some care this can be done so that 
$H$ is the stabiliser of this wall.  At this point, Wise 
applies Sageev's construction and, because of the above properties of antennas, 
is able to show that any axis in $\widetilde C$ for an element of $F$ is cut by 
some translate of $\widetilde W$.  Hence the action on the dual cube complex 
$\Psi$ is proper, and it is cocompact for general reasons~\cite{Sageev97}.  It 
is also not hard to see that the action of $F$ on $\Psi$ is essential.

Let $D$ be a hyperplane of $\Psi$, which, by construction, has stabiliser $H$.  Since $H$ stabilised the two halfspaces associated to the wall $\widetilde W$, the action of $F$ on $\Psi$ is without inversions in hyperplanes (there is no need to subdivide).

\textbf{Finding an extremal panel in $\Psi$:}  We can and shall insist that the 
$+$ and $-$ were assigned so that each tree attached to $\bar C$ (as opposed to 
the antenna), if any, is assigned $+$.   

We will prove that $H$ does not 
act essentially on $D$; the proof of Corollary~\ref{cor:lower_complexity} will 
then show that $\Psi$ has an extremal panel and panel collapse is in play.

Let $\widetilde W\subset\widetilde C$ be a lift of the universal cover of $W$, 
so that $\widetilde W$ consists of $\widetilde{\bar C}$ together with an 
antenna $hA$ based at $h\tilde v$, for each $h\in H$, where $\tilde v$ is a 
lift of $v$ in $\widetilde{\bar C}$ and $A$ is the lift of the antenna at that 
point.  

The construction of $A$ ensures that there exists $g\in F-H$ so that the distinct walls $\widetilde W,g\widetilde W$ intersect in the following way:
\begin{itemize}
 \item $\widetilde W\cap g\widetilde W= A\cap gA$;
 \item $gA\cap A=gT_i$ for some $i$, and $gT_i\subset P$.
\end{itemize}

Let $\widetilde W^+,\widetilde W^-$ denote the subtrees of $\widetilde C$ projecting to the subgraphs of $\ddot C$ labelled 
$+$ and $-$ respectively.  Then (up to relabelling), $\widetilde W^+\cap g\widetilde W^+$ contains $\widetilde{\bar C}$ and 
$g\widetilde{\bar C}$, while $\widetilde W^+\cap g\widetilde W^-$,  $\widetilde W^-\cap g\widetilde W^-$, $\widetilde W^-\cap 
g\widetilde W^+$ are nonempty and bounded.  

Hence the walls in $\widetilde C$ determined by $\widetilde W,g\widetilde W$ \emph{cross} (i.e. all four possible intersections of associated halfspaces are nonempty), so, in the dual cube complex $\Psi$, the hyperplanes $D,gD$ corresponding to these walls also cross.  On the other hand, each of $D,gD$ is extremal in the other.  Indeed, if $gD\cap D$ were an essential hyperplane of $D$ and vice versa, then all four of the possible intersections of the halfspaces $\widetilde W^\pm,g\widetilde W^\pm$ would be unbounded.  Hence $H$ does not act essentially on $D$ and, as explained in the proof of Corollary~\ref{cor:lower_complexity}, $\Psi$ has an extremal panel $P$.
\end{proof}

In fact Wise's antenna construction allows us to perform a finite
sequence of panel collapses that reduces the original cubulation to a
free action on a tree. In contrast to Proposition~\ref{prop:antenna},
using multiple hyperplanes, it seems one can construct geometric
actions of $F$ on a CAT(0) cube complex $\Psi$ so that $\Psi$ has no
compact hyperplanes and every hyperplane stabiliser acts essentially
on its hyperplane.

\subsection{The Cashen-Macura complexes}\label{sec:cashen_macura}

\begin{defn}\label{defn:line-pattern}
  Let $X$ be a Gromov hyperbolic space. A \emph{line pattern}
  $\mathcal L$ in $X$ is a collection of quasi-isometry classes of
  bi-infinite quasi-geodesics in $X$. Given spaces $X_i,\mathcal L_i$,
  equipped with line patterns, for $i=1,2$ we say that a
  quasi-isometry $\phi:X_1\to X_2$ \emph{respects the line patterns}
  if for any representative $l\in \mathcal L_1$ the mapping
  $l \mapsto \phi(l)$ is to a representative of an element of
  $\mathcal L_2$ and induces a bijection
  $\mathcal L_1 \stackrel{\sim}{\to} \mathcal L_2$.
\end{defn}

Two elements $h,g$ of a hyperbolic group $\Gamma$ are
\emph{commensurable} if there is some $k \in \Gamma$ such that
$[\langle g \rangle: \langle k^{-1}h k\rangle]\leq \infty$.

\begin{defn}
  If $\Gamma$ is a hyperbolic group and $\{g_1,\ldots,g_m\}$ is a
  tuple of non-commensurable hyperbolic elements, then the \emph{line
    pattern $\mathcal L$ generated by $\{g_1,\ldots,g_m\}$} is the
  collection of quasi-isometry classes of quasi-lines
  \[ \mathcal L = \{ h\cdot \langle g_i\rangle \mid h \in \Gamma,
    1\leq i \leq m \}/\textrm{(bounded distance)}.
  \] A line pattern is called \emph{rigid} if $\Gamma$ admits no
  virtually cyclic splittings relative to $\{g_1,\ldots,g_m\}$;
  equivalently, if for every cut pair ${c,c'} \subset \partial \Gamma$
  there is some $l \in \mathcal L$ with endpoints
  $l^+,l^- \in \partial \Gamma$ such that $l^+$ is in one connected
  component of $\partial \Gamma \setminus \{c,c'\}$ and $l^-$ is in
  a different component, or finally the virtually cyclic JSJ
  decomposition relative to $\{g_1,\ldots,g_m\}$ is trivial.
\end{defn}

In \cite{cashen-macura} the following remarkable \emph{pattern
  rigidity} theorem, which promotes certain quasi isometries to action
by isometries on a common space, is proved.
\begin{thm}[Pattern rigidity, see {\cite[Theorem
    5.5]{cashen-macura}}]\label{thm:cashen-macura}
  Let $F_i,\mathcal L_i,i=0,1$, be free groups equipped with a rigid
  line patterns. Then there are CAT(0) cube complexes $X_i$ with line
  patterns $\mathcal L_i$ and embeddings
  \begin{equation}\label{eqn:embedding}
    F_i\stackrel{\iota_i}{\hookrightarrow} \mathrm{Isom}(X_i)
  \end{equation}
  inducing cocompact isometric actions $F_i\circlearrowright X_i$,
  which in turn induce equivariant line pattern preserving
  quasi-isometries \[ \phi_i:F_i \to X_i.\] Furthermore for any line
  pattern preserving quasi-isometry $q:F_0\to F_1$ there is a line
  pattern preserving isometry $\alpha_q$ such that the following
  diagram of line pattern preserving quasi-isometries commutes up to
  bounded distance:\[
    \begin{tikzpicture}[scale=1.5]
      \node (XL1) at (-1,0) {($X_0,\mathcal L_0)$};
      \node (XL2) at (1,0) {($X_1,\mathcal L_1)$};
      \node (F1L) at (-1,-1) {($F_0,\mathcal L_0)$};
      \node (F2L) at (1,-1) {($F_1,\mathcal L_1)$};
      \draw[->] (F1L) --node[left]{$\phi_0$} (XL1);
      \draw[->] (F2L) --node[right]{$\phi_1$} (XL2);
      \draw[->] (F1L) --node[above]{$q$} (F2L);
      \draw[->] (XL1) --node[above]{$\alpha_q$} (XL2);
    \end{tikzpicture}.
  \]
  We note that the actions of $F_i$ on $X_i$ are free since the
  quasi-isometries $\phi_i:F_i \to X_i$ are equivariant.
\end{thm}

In \cite[\S 6.4]{cashen-macura} the authors give an example where the
\emph{Cashen-Macura complex} $X$ given in Theorem
\ref{thm:cashen-macura} is not a tree, but observe that it is possible
in this example to construct another complex $X'$ satisfying the
requirement of the theorem that is a tree. They ask if this can always
be done. While not precisely answering their question, which is about a
choice of topologically distinguished cut sets, we have the following.

\begin{thm}\label{prop:cashen-macura-tree}
  Given a free group equipped with a line pattern $F_0,\mathcal L_0$,
  there exists a locally finite tree equipped with line pattern
  $T,\mathcal L_T$ which satisfies the requirements of the
  Cashen-Macura complex $X,\mathcal L$ given in Theorem
  \ref{thm:cashen-macura}. Furthermore $\mathcal L_T$ can be
  represented by geodesics.
\end{thm}
\begin{proof}
  In \cite[\S 5.2.2]{cashen-macura} it is shown that in the CAT(0)
  cube complex $X$ constructed in Theorem \ref{thm:cashen-macura}
  there is a uniform bound on the number of hyperplanes a given
  hyperplane $H$ crosses and $X$ is shown to be locally finite. It
  follows that $X$ has compact hyperplanes. Let
  $G = \mathrm{Isom}(X)$. By repeatedly applying Corollary
  \ref{cor:lower_complexity} there a $G$-equivariant deformation
  retraction from $X$ to a tree $T$. We can take $\mathcal L_T$ to be
  the collection of geodesic representatives of the quasi-lines of
  $\mathcal L$ in $T$. $(T,\mathcal L_T)$, by definition, satisfy the
  necessary requirements.
\end{proof}

Ultimately, given a line pattern preserving quasi-isometry
$q:(F_0,\mathcal L_0)\to(F_1,\mathcal L_1)$, we would like to
construct a virtual isomorphism between $F_0$ and $F_1$ that induces a
line pattern preserving quasi-isometry. The appropriate machinery to
obtain this result appears to be commensurability of tree lattices:
let $T$ be a tree equipped with a line pattern $\mathcal L$ and denote
$G = \mathrm{Isom}(T)$ and $G_{\mathcal L}$ the subgroup of isometries
of $T$ that preserve $\mathcal L$. Then the existence of such a
virtual isomorphism reduces to asking whether the images of $F_0$ and
$F_1$ in $G$ are commensurable within the subgroup $G_{\mathcal L}$ of
line pattern preserving isometries.

There are examples when $G_{\mathcal L}$ is a discrete group (see
\cite[\S 6.2]{cashen-macura}) and therefore finitely generated. In
this case the embeddings given in (\ref{eqn:embedding}) of Theorem
\ref{thm:cashen-macura} would give virtual isomorphisms,
i.e. isomorphisms between finite index subgroups, between $F_0$ and
$F_1$ that preserve line patterns. However in \cite[\S
6.3]{cashen-macura} an example shows that even if $\mathcal L$ is a
rigid line pattern, $G_{\mathcal L}$ may not be a discrete group of
automorphisms of $T$. In this case, even if $F_0$ and $F_1$ act
cocompactly on $T$, their intersections in $G_{\mathcal L}$ could be
trivial.

If $H$ is a closed subgroup of $\mathrm{Aut(T)}$, then a discrete
subgroup $\Gamma \subset H$ with the same orbits as $H$, i.e. with
$H\backslash T = \Gamma\backslash T$ is called a \emph{discrete
  grouping.} Discrete groupings are hard to obtain, which is why being
able to make Cashen-Macura complexes into trees is advantageous: we
are now able to apply the tree lattice techniques in
\cite{bass1990uniform}. We finish with an observation that may be
useful, in constructing virtual isomorphisms.

\begin{prop}[$G_{\mathcal L}$ is closed and admits a discrete
  grouping.]\label{prop:discrete-grouping-I}
  Let $F_0,\mathcal L_0$ be a free group equipped with a rigid line
  pattern.  Then the group $G_{\mathcal L}$ of automorphisms which
  preserve the line pattern $\mathcal L$ of the Cashen-Macura tree
  $T,\mathcal L$, given in Proposition \ref{prop:cashen-macura-tree},
  is a closed subgroup of $\mathrm{Aut}(T)$ and admits a finitely
  generated, discrete grouping $\Phi^0 \leq G_\mathcal{L}$.
\end{prop}

Unfortunately we are currently unable to obtain a commensurability
result as we do not know if there is a discrepancy between
$G_{\mathcal L}$ and $G_{G_{\mathcal L}}$ (see \cite[Theorem 4.7
(iv)]{bass1990uniform} or \cite[Theorem 3]{lim_covering_2008}.)

\begin{proof}
  Given $F_0,\mathcal L_0$, let $T,\mathcal L$ be the Cashen-Macura
  tree constructed in Proposition \ref{prop:cashen-macura-tree}. Let
  $G = \mathrm{Aut}(T)$ and let $G_{\mathcal L}$ denote the subgroup
  that preserves the geodesics in $\mathcal L$.

  \textbf{$G_{\mathcal L}$ is a closed subgroup of $G$.} Equip $G$
  with the compact open topology. Let
  $\gamma \in G \setminus G_{\mathcal L}$. Suppose first that there is
  some line $l \in \mathcal L$ such that
  $\gamma \cdot l \not \in \mathcal L$.  $\gamma\cdot l$ is still a
  geodesic in $T$. Because there is a uniform bound on the
  intersection of any two distinct lines in $\mathcal L$, it must
  follow that there is some finite subset $l_0\subset l$ such that
  $\gamma\cdot l_0$ is not contained within any $l' \in \mathcal
  L$. Let $B\subset T$ be a metric ball containing $l_0$. Then the set
  $U$ of isometries that coincide with $\gamma$ on $B$ give an open
  neighbourhood
  $\gamma \in U \subset \left(G \setminus G_{\mathcal L}\right)$
  separating $\gamma$ from $G_{\mathcal L}$.

  Next suppose that for every $l \in \mathcal L$,
  $\gamma\cdot l \in \mathcal L$ but that
  $\gamma \cdot \mathcal L \subsetneq \mathcal L$. Let
  $l \in \mathcal L \setminus \gamma\cdot\mathcal L$, then
  $\gamma^{-1}\cdot l \not\in \mathcal L$. By the argument of the
  previous paragraph there is an open set
  $\gamma^{-1} \in U \subset \left(G \setminus G_{\mathcal L}\right)$
  separating $\gamma^{-1}$ from $G_{\mathcal L}$. Since $G$ is a
  topological group, the inversion operation is a homeomorphism
  $-^{-1}:G \to G$ that maps $G_{\mathcal L}$ to itself, since the
  latter is a subgroup. The image $U^{-1}$ of $U$ gives a open
  neighbourhood
  $\gamma \in U^{-1} \subset \left(G \setminus G_{\mathcal L}\right)$
  separating $\gamma$ from $G_{\mathcal L}$.  It follows that
  $G_{\mathcal L}$ is closed in $G$.

  \textbf{Applying tree lattice techniques.}  For any subgroup of $G$ there is a well defined homomorphism to
  $\mathbb Z_2$ whose kernel does not invert any edges (see \cite[\S
  3]{bass1990uniform} or \cite[\S 6.3]{bass}.) If necessary we
  therefore pass to index 2 subgroups that do not invert edges of $T$,
  but keep our notation.  By Theorem \ref{thm:cashen-macura} the group
  $F_0$ acts freely on $T$. By hypothesis the action is also
  cocompact. $F_0\backslash T$ is therefore a finite graph. It follows
  that $F_0 \subset G_{\mathcal L}$ is a \emph{uniform lattice} in the
  sense of \cite[Definition 4.3]{bass1990uniform}. Because
  $G_{\mathcal L}$ is closed, \cite[Proposition 4.5]{bass1990uniform}
  then implies that $G_{\mathcal L}$ is \emph{unimodular} and
  \cite[Theorem 4.7]{bass1990uniform} implies the existence of a
  discrete subgroup $\Phi^0 \subset G_{\mathcal L}$ with
  $\Phi^0\backslash T = G_{\mathcal L}\backslash T$.
  
  Since $G_{\mathcal L}$ contains $F_0$, and $F_0$ acts cocompactly on
  $T$, it follows that $G_{\mathcal L}$, and hence $\Phi^0$, acts
  cocompactly on $T$.  Since $\Phi^0$ is discrete, stabilisers in
  $\Phi^0$ of points in the locally finite tree $T$ are finite, so
  $\Phi^0$ acts on $T$ properly.  Hence $\Phi^0$ is finitely
  generated, by the Milnor-\v{S}varc lemma.
\end{proof}

We end this section with a question that we hope would be of interest to tree
lattice experts.

\begin{question}
  Let $\mathcal L$ be a rigid line pattern in a tree $T$ and suppose
  that $G_{\mathcal L} \leq \mathrm{Aut}(T)$ is closed and
  unimodular. Can there be a proper
  inclusion\[ G_{\mathcal L} \lneq G_{G_{\mathcal L}},
  \] where $G_{G_{\mathcal L}} \leq \mathrm{Aut}(T)$ is the maximal
  group with \[
    G_{\mathcal L}\backslash T = G_{G_{\mathcal L}}\backslash T?
    \]
\end{question}

In \cite{woodhouse_revisiting_2018}, Daniel Woodhouse solved the
problem that was the original motivation for this question.

\subsection{The Kropholler conjecture}\label{subsec:kropholler}
We now apply Corollary~\ref{cor:lower_complexity} to the following special case of the Kropholler conjecture.

\begin{cor}[Kropholler conjecture, cocompact case]\label{cor:kropholler}
Let $G$ be a finitely generated group and $H\leq G$ a subgroup with $e(G,H)\geq2$. Let $\Psi$ be the dual cube complex associated to the pair $(G,H)$, so that $\Psi$ has one $G$--orbit of hyperplanes and each hyperplane stabiliser is a conjugate of $H$.  Suppose that:
\begin{itemize}
 \item $G$ acts on $\Psi$ cocompactly;
 \item $H$ acts with a global fixed point on the associated hyperplane.
\end{itemize}
Then $G$ admits a nontrivial splitting over a subgroup commensurable with a subgroup of $H$.
\end{cor}

\begin{proof}
Let $D$ be a hyperplane of $\Psi$ with $\stabilizer_G(D)=H$.  

\textbf{Bounded hyperplanes:}  The fixed-point hypothesis guarantees that the action of $H$ on $D$ is non-essential, so we could apply Corollary~\ref{cor:lower_complexity} immediately.  In fact, since $H$ acts cocompactly on $D$, and also fixes a point in $D$, we have that $D$ has finite diameter.

\textbf{Applying Corollary~\ref{cor:lower_complexity}:}  Since $D$ is bounded, Corollary~\ref{cor:compact_hyp} provides an extremal panel in $\Psi$.  At this point, we can subdivide $\Psi$ once if necessary to ensure that the action of $G$ is without inversions.  This has the effect of replacing $\stabilizer_G(D)$ with a subgroup of index at most $2$, which will not affect the conclusion.  (As usual, we will not subdivide at later stages in the induction, but instead use that the no inversions property of the action persists under panel collapse.)

Corollary~\ref{cor:lower_complexity} provides a cocompact, inversion-free $G$--action on a CAT(0) cube complex $\Psi_\bullet$ with $\#(\Psi_\bullet)<\#(\Psi)$, unless $\Psi$ was already a tree.  Moreover, the hyperplanes of $\Psi_\bullet$ are components of $\Psi_\bullet\cap gD$ for $g\in G$, and are thus bounded. Furthermore this intersection property implies that, since $G$ acts by permuting the hyperplanes in $\Psi$, a $\Psi_\bullet$--hyperplane stabilizer must lie in the stabilizer of an intersection of $\Psi$--hyperplanes. It follows that each stabiliser of a hyperplane in $\Psi_\bullet$ is virtually contained in a stabiliser of a hyperplane in $\Psi$ and the index is bounded by the dimension of $\Psi$. It thus follows by induction on complexity that $G$ acts on a simplicial tree all of whose edges has stabiliser commensurable with a subgroup of $H$, as required.
\end{proof}

\bibliographystyle{acm}
\bibliography{biblio.bib}

\def\cprime{$'$}
\begin{thebibliography}{10}

\bibitem{bass}
{\sc Bass, H.}
\newblock Covering theory for graphs of groups.
\newblock {\em J. Pure Appl. Algebra 89}, 1-2 (1993), 3--47.

\bibitem{bass1990uniform}
{\sc Bass, H., and Kulkarni, R.}
\newblock Uniform tree lattices.
\newblock {\em Journal of the American Mathematical Society 3}, 4 (1990),
  843--902.

\bibitem{BHS:I}
{\sc Behrstock, J., Hagen, M., and Sisto, A.}
\newblock Hierarchically hyperbolic spaces, {I}: Curve complexes for cubical
  groups.
\newblock {\em Geometry \& Topology 21}, 3 (2017), 1731--1804.

\bibitem{BergeronWise}
{\sc Bergeron, N., and Wise, D.~T.}
\newblock A boundary criterion for cubulation.
\newblock {\em American Journal of Mathematics 134}, 3 (2012), 843--859.

\bibitem{bridson2011metric}
{\sc Bridson, M.~R., and Haefliger, A.}
\newblock {\em Metric spaces of non-positive curvature}, vol.~319.
\newblock Springer Science \& Business Media, 2011.

\bibitem{CapraceSageev}
{\sc Caprace, P.-E., and Sageev, M.}
\newblock Rank rigidity for {CAT}(0) cube complexes.
\newblock {\em Geometric and functional analysis 21}, 4 (2011), 851--891.

\bibitem{cashen-macura}
{\sc Cashen, C.~H., and Macura, N.}
\newblock Line patterns in free groups.
\newblock {\em Geometry \& Topology 15}, 3 (2011), 1419--1475.

\bibitem{CharneyDavis}
{\sc Charney, R., and Davis, M.~W.}
\newblock Finite {$K(\pi,1)$}s for {A}rtin groups.
\newblock {\em Prospects in Topology (Princeton, NJ, 1994), Princeton
  University Press, Princeton\/} (1995), 110--124.

\bibitem{chatterji2005wall}
{\sc Chatterji, I., and Niblo, G.}
\newblock From wall spaces to {$\rm CAT(0)$} cube complexes.
\newblock {\em Internat. J. Algebra Comput. 15}, 5-6 (2005), 875--885.

\bibitem{Culler:nielsen}
{\sc Culler, M.}
\newblock Finite groups of outer automorphisms of a free group.
\newblock {\em Contributions to group theory 33\/} (1984), 197--207.

\bibitem{Kron}
{\sc Dunwoody, M., and Kr{\"o}n, B.}
\newblock Vertex cuts.
\newblock {\em Journal of Graph Theory 80}, 2 (2015), 136--171.

\bibitem{DunwoodyRoller}
{\sc Dunwoody, M., and Roller, M.}
\newblock Splitting groups over polycyclic-by-finite subgroups.
\newblock {\em Bulletin of the London Mathematical Society 25}, 1 (1993),
  29--36.

\bibitem{Dunwoody-fg}
{\sc Dunwoody, M.~J.}
\newblock Cutting up graphs.
\newblock {\em Combinatorica 2}, 1 (1982), 15--23.

\bibitem{dunwoody-1985}
{\sc Dunwoody, M.~J.}
\newblock The accessibility of finitely presented groups.
\newblock {\em Invent. Math. 81}, 3 (1985), 449--457.

\bibitem{dunwoody2017structure}
{\sc Dunwoody, M.~J.}
\newblock Structure trees, networks and almost invariant sets.
\newblock In {\em Groups, graphs and random walks}, vol.~436 of {\em London
  Math. Soc. Lecture Note Ser.} Cambridge Univ. Press, Cambridge, 2017,
  pp.~137--175.

\bibitem{evangelidou2014cactus}
{\sc Evangelidou, A., and Papasoglu, P.}
\newblock A cactus theorem for end cuts.
\newblock {\em International Journal of Algebra and Computation 24}, 01 (2014),
  95--112.

\bibitem{gerasimov1997semi}
{\sc Gerasimov, V.~N.}
\newblock Semi-splittings of groups and actions on cubings.
\newblock {\em Algebra, geometry, analysis and mathematical physics
  (Russian)(Novosibirsk, 1996)\/} (1997), 91--109.

\bibitem{HagenPrzytycki}
{\sc Hagen, M.~F., and Przytycki, P.}
\newblock Cocompactly cubulated graph manifolds.
\newblock {\em Israel Journal of Mathematics 207}, 1 (2015), 377--394.

\bibitem{HagenWilton}
{\sc Hagen, M.~F., and Wilton, H.}
\newblock Guirardel cores for multiple cube complex actions.
\newblock {\em In preparation\/} (2017).

\bibitem{HagenWise}
{\sc Hagen, M.~F., and Wise, D.~T.}
\newblock Cubulating hyperbolic free-by-cyclic groups: the general case.
\newblock {\em Geometric and Functional Analysis 25}, 1 (2015), 134--179.

\bibitem{Haglund:semisimple}
{\sc Haglund, F.}
\newblock Isometries of {CAT}(0) cube complexes are semi-simple.
\newblock {\em arXiv preprint arXiv:0705.3386\/} (2007).

\bibitem{kapovich2014energy}
{\sc Kapovich, M.}
\newblock Energy of harmonic functions and {G}romov's proof of {S}tallings'
  theorem.
\newblock {\em Georgian Mathematical Journal 21}, 3 (2014), 281--296.

\bibitem{KarNiblo}
{\sc Kar, A., and Niblo, G.~A.}
\newblock Relative ends, {$\ell_2$}-invariants and property ({T}).
\newblock {\em Journal of Algebra 333}, 1 (2011), 232--240.

\bibitem{Kropholler}
{\sc Kropholler, P.}
\newblock An analogue of the torus decomposition theorem for certain
  poincar{\'e} duality groups.
\newblock {\em Proceedings of the London Mathematical Society 3}, 3 (1990),
  503--529.

\bibitem{LauerWise}
{\sc Lauer, J., and Wise, D.~T.}
\newblock Cubulating one-relator groups with torsion.
\newblock In {\em Mathematical Proceedings of the Cambridge Philosophical
  Society\/} (2013), vol.~155, Cambridge University Press, pp.~411--429.

\bibitem{lim_covering_2008}
{\sc Lim, S., and Thomas, A.}
\newblock Covering theory for complexes of groups.
\newblock {\em Journal of Pure and Applied Algebra 212}, 7 (July 2008),
  1632--1663.

\bibitem{Martin:Higman}
{\sc Martin, A.}
\newblock On the cubical geometry of {H}igman's group.
\newblock {\em Duke Mathematical Journal 166}, 4 (2017), 707--738.

\bibitem{niblo2004geometric}
{\sc Niblo, G.~A.}
\newblock A geometric proof of {S}tallings' theorem on groups with more than
  one end.
\newblock {\em Geometriae Dedicata 105}, 1 (2004), 61--76.

\bibitem{NibloReeves}
{\sc Niblo, G.~A., and Reeves, L.~D.}
\newblock Coxeter groups act on {CAT}(0) cube complexes.
\newblock {\em Journal of Group Theory 6}, 3 (2003), 399.

\bibitem{NibloSageev}
{\sc Niblo, G.~A., and Sageev, M.}
\newblock On the {K}ropholler conjecture. {I}n \emph{Guido's book of
  conjectures}.
\newblock vol.~40 of {\em Monographies de L'Enseignement Math\'ematique
  [Monographs of L'Enseignement Math\'ematique]}. L'Enseignement
  Math\'ematique, Geneva, 2008, p.~189.
\newblock A gift to Guido Mislin on the occasion of his retirement from ETHZ
  June 2006, Collected by Indira Chatterji.

\bibitem{nica2004cubulating}
{\sc Nica, B.}
\newblock Cubulating spaces with walls.
\newblock {\em Algebr. Geom. Topol 4\/} (2004), 297--309.

\bibitem{OllivierWise}
{\sc Ollivier, Y., and Wise, D.}
\newblock Cubulating random groups at density less than 1/6.
\newblock {\em Transactions of the American Mathematical Society 363}, 9
  (2011), 4701--4733.

\bibitem{sageev1995ends}
{\sc Sageev, M.}
\newblock Ends of group pairs and non-positively curved cube complexes.
\newblock {\em Proceedings of the London Mathematical Society 3}, 3 (1995),
  585--617.

\bibitem{Sageev97}
{\sc Sageev, M.}
\newblock Codimension-1 subgroups and splittings of groups.
\newblock {\em Journal of Algebra 189}, 2 (1997), 377--389.

\bibitem{Sageev:pcmi}
{\sc Sageev, M.}
\newblock {$\rm CAT(0)$} cube complexes and groups.
\newblock In {\em Geometric group theory}, vol.~21 of {\em IAS/Park City Math.
  Ser.} Amer. Math. Soc., Providence, RI, 2014, pp.~7--54.

\bibitem{stallings1968torsion}
{\sc Stallings, J.~R.}
\newblock On torsion-free groups with infinitely many ends.
\newblock {\em Annals of Mathematics\/} (1968), 312--334.

\bibitem{stallings-gt3dm}
{\sc Stallings, J.~R.}
\newblock {\em Group theory and three-dimensional manifolds}.
\newblock Yale University Press New Haven, 1971.

\bibitem{touikan2015one}
{\sc Touikan, N.}
\newblock On the one-endedness of graphs of groups.
\newblock {\em Pacific Journal of Mathematics 278}, 2 (2015), 463--478.

\bibitem{Whitehead:i}
{\sc Whitehead, J. H.~C.}
\newblock Combinatorial homotopy. {I}.
\newblock {\em Bulletin of the American Mathematical Society 55}, 3 (1949),
  213--245.

\bibitem{Whitehead:ii}
{\sc Whitehead, J. H.~C.}
\newblock Combinatorial homotopy. {II}.
\newblock {\em Bulletin of the American Mathematical Society 55}, 5 (1949),
  453--496.

\bibitem{WiseTubular}
{\sc Wise, D.}
\newblock Cubular tubular groups.
\newblock {\em Transactions of the American Mathematical Society 366}, 10
  (2014), 5503--5521.

\bibitem{Wise:small_can}
{\sc Wise, D.~T.}
\newblock Cubulating small cancellation groups.
\newblock {\em Geometric \& Functional Analysis GAFA 14}, 1 (2004), 150--214.

\bibitem{Wise:RR}
{\sc Wise, D.~T.}
\newblock {\em From riches to {RAAGS}: 3-manifolds, right-angled Artin groups,
  and cubical geometry}, vol.~117.
\newblock American Mathematical Soc., 2012.

\bibitem{Wise:recube}
{\sc Wise, D.~T.}
\newblock Recubulating free groups.
\newblock {\em Israel Journal of Mathematics 191}, 1 (2012), 337--345.

\bibitem{woodhouse_revisiting_2018}
{\sc Woodhouse, D.~J.}
\newblock Revisiting {Leighton}'s {Theorem} with the {Haar} {Measure}.
\newblock {\em arXiv:1806.08196 [math]\/} (June 2018).
\newblock arXiv: 1806.08196.

\end{thebibliography}

\end{document}